\documentclass[12pt]{amsart}
\usepackage{graphicx}
\usepackage{mathrsfs}
\usepackage{epsfig}
\usepackage{amsmath}
\usepackage{amsfonts}
\usepackage{amssymb}
\usepackage{amssymb,amscd}
\usepackage{tikz}
\usepackage{verbatim}
\usepackage{booktabs}
\usepackage[all]{xy}
\usepackage{mathtools}
\usepackage{overpic}
\usepackage{hyperref}
\usepackage{geometry}
\usepackage{hyperref}

\usepackage{amsthm}
\usepackage{amssymb}
\usepackage{amsfonts}
\usepackage{amsmath}
\usepackage{mathtools}
\usepackage[shortlabels]{enumitem}
\usepackage{adjustbox}
\usepackage{url} 
\usepackage[T1]{fontenc}

\usepackage{float}

\DeclareFontFamily{U}{mathx}{}
\DeclareFontShape{U}{mathx}{m}{n}{<-> mathx10}{}
\DeclareSymbolFont{mathx}{U}{mathx}{m}{n}
\DeclareMathAccent{\widehat}{0}{mathx}{"70}
\DeclareMathAccent{\widecheck}{0}{mathx}{"71}
\allowdisplaybreaks[4]
\newtheorem{thm}{Theorem}[section]

\newtheorem{lem}[thm]{Lemma}
\newtheorem{prop}[thm]{Proposition}
\newtheorem{conj}[thm]{Conjecture}
\theoremstyle{definition}
\newtheorem{defn}[thm]{Definition}
\theoremstyle{remark}
\newtheorem{rem}[thm]{Remark}
\numberwithin{equation}{section}

\renewcommand{\leq}{\leqslant}
\renewcommand{\geq}{\geqslant}

\newcommand{\hc}{\mathbf{H}^2_{\mathbb{C}}}

\begin{document}
	
	\date{\today}

	\title[]{Complex hyperbolic geometry of  Chain links}
	\author[J. Ma]{Jiming Ma}
	\address{School of Mathematical Sciences, Fudan University, Shanghai, China}
	\email{majiming@fudan.edu.cn}
	
	\author[B. Xie]{Baohua Xie}
	\address{School of Mathematics, Hunan University, Changsha, China}
	\email{xiexbh@hnu.edu.cn}
	
		\author[M. Xu]{Mengmeng Xu}
	\address{School of Mathematical Sciences, Fudan University, Shanghai, China}
	\email{mengmg.xu@gmail.com}
	
	
	
	\thanks{Jiming Ma was  supported by NSFC (No.12171092), he is also a member of LMNS, Fudan University. Baohua Xie was supported by NSFC (No.11871202).}
	
    \begin{abstract}
    The complex hyperbolic triangle group
    $\Gamma=\Delta_{4,\infty,\infty;\infty}$
    acting on the complex hyperbolic plane ${\bf H}^2_{\mathbb C}$
    is generated by complex reflections $I_1$, $I_2$, $I_3$ such that
    the product $I_2I_3$ has order four, 
    the products $I_3I_1$, $I_1I_2$ are parabolic
    and the product $I_1I_3I_2I_3$ is an accidental parabolic element.
    Unexpectedly, the product $I_1I_2I_3I_2$ is a hidden accidental parabolic element.
    We show that the 3-manifold at infinity of $\Delta_{4,\infty,\infty;\infty}$ is the complement of the chain link $8^4_1$ in the 3-sphere. 
    In particular, the quartic cusped hyperbolic 3-manifold $S^3-8^4_1$ admits a spherical CR-uniformization.  
    The proof relies on a new technique to show that the ideal boundary of the Ford domain is an infinite-genus handlebody. 
    
    Motivated by this result and supported by the previous studies of various authors, we conjecture that the chain link $C_p$ is an ancestor of the 3-manifold at infinity of the critical complex hyperbolic triangle group $\Delta_{p,q,r;\infty}$, for $3 \leq p \leq 9$.

	
    \end{abstract}

     \maketitle
	
\section{Introduction}

\subsection{Motivation}
The relationship between the topology of a manifold and the geometry that the manifold admits is a classic topic. 
For example, Thurston's work on 3-manifolds has shown that geometry plays an important role in the study of the topology of 3-manifolds.

A {\it spherical CR-structure} on a 3-manifold $M$ is a maximal collection of distinguished charts modelled on the boundary $\partial \mathbf{H}^2_{\mathbb C}$
of the complex hyperbolic plane $\mathbf{H}^2_{\mathbb C}$, where the transition maps are elements in $\mathbf{PU}(2,1)$.
In other words, a {\it spherical CR-structure} is a $(G,X)$-structure with $G=\mathbf{PU}(2,1)$ and $X=S^3$. 
In general, it is difficult to determine whether or not a 3-manifold admits a spherical CR-structure. 
In contrast to other geometric structures admitted by 3-manifolds, there are relatively few known examples admitting spherical CR-structures.

We focus on the following important class of spherical CR-structures: a spherical CR-structure on a 3-manifold $M$  is called {\it uniformizable} if it is
given by $M=\Gamma\backslash \Omega_{\Gamma}$, where $\Omega_{\Gamma}\subset \partial \mathbf{H}^2_{\mathbb C}$ is the discontinuity region of a discrete subgroup $\Gamma < \mathbf{PU}(2,1)$ acting on $\partial \mathbf{H}^2_{\mathbb C}=S^3$.
Thus, the study of discrete subgroups of  $\mathbf{PU}(2,1)$ will be crucial to understand uniformizable spherical CR-structures.
Complex hyperbolic triangle groups provide rich examples of such discrete subgroups.

Let $T_{p,q,r}$ be the abstract reflection triangle group with the presentation
$$\langle \sigma_1, \sigma_2, \sigma_3 | \sigma^2_1=\sigma^2_2=\sigma^2_3=(\sigma_2 \sigma_3)^p=(\sigma_3 \sigma_1)^q=(\sigma_1 \sigma_2)^r=id \rangle,$$
where each of $p,q,r$ is either a positive integer or $\infty$ and $1/p+1/q+1/r<1$. 
One can assume that $p \leq q \leq r$. 
If some $p,q$ and $r$ is equal to $\infty$, then
the corresponding relation vanishes.  
The ideal triangle group is one with $p=q=r=\infty$.  
A $(p,q,r)$-\emph{complex hyperbolic triangle group} is a representation $\rho: T_{p,q,r}\longrightarrow\mathbf{PU}(2,1)$
where the image of each generator is a complex reflection that fixes a complex line. 
It is well known \cite{Pratoussevitch:2005}  that the space of $(p,q,r)$-complex hyperbolic triangle groups has real dimension one when $3 \leq p \leq q \leq r$.  
Therefore, one can parameterize the family of complex hyperbolic triangle groups by $\rho_{t}:T_{p,q,r} \rightarrow \mathbf{PU}(2,1)$ for $t \in [0,1]$. 
Here $\rho_{0}$ corresponds to the $\mathbb{R}$-Fuchsian case, that is, $\rho_{0}(T_{p,q,r})$ stabilizes a $\mathbb{R}$-plane in ${\bf H}^2_{\mathbb C}$.
We denote the image group of a $(p,q,r)$-complex hyperbolic triangle group by $\Delta_{p,q,r}$, and set $\rho_t(\sigma_i)=I_i$.
We further denote the image group by $\Delta_{p,q,r;\infty}$ when $I_1 I_3I_2 I_3$ is parabolic.

Goldman and Parker first discussed the complex hyperbolic deformation of the ideal triangle group in \cite{GoPa}.
They gave an interval in the moduli space of complex hyperbolic ideal triangle groups whose points correspond to discrete and faithful representations.
They conjectured that a complex hyperbolic ideal triangle group $\Delta_{\infty,\infty, \infty}=\langle I_1, I_2, I_3 \rangle$ is discrete and faithful if and only if $I_1 I_2 I_3$ is not elliptic.  
Schwartz proved  Goldman-Parker's conjecture in \cite{Schwartz:2001ann, Schwartz:2006}.    
Furthermore,  Schwartz \cite{Schwartz:2001acta} analyzed the complex hyperbolic ideal triangle group $G$ when $I_1 I_2 I_3$ is parabolic, and showed that the 3-manifold at infinity of the quotient space $\hc/G$ is commensurable with the Whitehead link complement in the 3-sphere. 
In particular, the Whitehead link complement admits a uniformizable spherical CR-structure.

In general, for $\Delta_{p,q,r}= \rho_t(T_{p,q,r})$ with $\rho_t(\sigma_i)=I_i$, set $$W_A=I_1I_3I_2I_3 \quad  \text{and} \quad W_B=I_1I_2I_3.$$
Schwartz has also conjectured the following.

\begin{conj}[Schwartz \cite{schwartz-icm}] The set of $t \in [0,1]$ for which $\rho_{t}$ is a discrete and faithful representation of $T_{p,q,r}$ is the closed interval consisting of the parameters $t$ for which neither $W_A$ nor $W_B$ is elliptic. 
\end{conj}

The interval in this conjecture is called the \emph{critical interval}. 
The non-zero end-point of the critical interval is called the \emph{critical point}.  
The group $T_{p,q,r}$ is of \emph{type $A$} if $W_A$ is parabolic in the representation that corresponds to the critical point. 
This means that $W_A$ becomes elliptic before $W_B$ in the moduli space. 
Otherwise, we say the group $T_{p,q,r}$ is of \emph{type $B$}.  
In particular, $T_{p,q,r}$ is of type $A$ if $p \leq 9$ and of type $B$ if $p\ge14$ \cite{schwartz-icm}. 
 Schwartz's conjecture has been proved in a few cases \cite{ParkerWX:2016, ParkerWill:2017}.

The representation at the critical point is very interesting. 
Note that there is exactly one $\mathbb{Z}_{3}$-covering of the Whitehead link complement which has four cusps, and we denote this manifold by $W$.  
If $p,q,r$	are all large enough, then $T_{p,q,r}$ is of type $B$. 
Moreover, Schwartz \cite{Schwartz:2007} showed that the 3-manifold at infinity of the even subgroup of the complex hyperbolic triangle group $\Delta_{p,q,r}$ at the critical point can be obtained from $W$ by Dehn fillings. 
However, for a general critical complex hyperbolic triangle group $\Delta_{p,q,r; \infty}$ of type $A$, the 3-manifold at infinity of it has not been identified yet.
Even so, people have found several more explicit examples of cusped hyperbolic 3-manifolds admitting uniformizable spherical CR-structures at these critical representations \cite{Acosta:2019, deraux-scr, der-fal, jwx,  MaXie2020, MaXie2021}. 	
Almost all of these explicit examples of uniformizable spherical CR-structures rely on difficult and sophisticated analysis. 
People do not know the topological/geometric reason why the 3-manifolds at infinity of these groups corresponding to critical points are what they got. 
Falbel, Guilloux and Will \cite{Falbel-Guilloux-Will} proposed the philosophy of predicting the 3-manifold when there is an accidental parabolic element.  
See \cite{MaXie2023} for the demonstration in a special case. 
With the main result and conjectures in Section \ref{sec:3mfdconj}, this paper appears to be the initial step in understanding all the 3-manifolds at infinity of complex hyperbolic triangle groups $\Delta_{p,q,r; \infty}$ of type $A$.

\subsection{Main result of the paper} In this paper, we prove the following.
\begin{thm} \label{thm:4pp}
Let $\Gamma=\langle I_1, I_2, I_3 \rangle $ be the critical complex hyperbolic  triangle group $\Delta_{4,\infty,\infty;\infty}$. 
Then the 3-manifold $M$ at infinity of the even subgroup $\langle I_1I_2,I_2I_3\rangle$ of  $\Gamma$ is the quartic cusped hyperbolic 3-manifold $S^3-8^4_1$, where $8^4_1$ is the chain link with four components as in Figure \ref{fig:841}.
\end{thm}

The link $8^4_1$ in Rolfsen's list \cite{Rolfsen} is the first link with crossing number eight and with four components.
In Hoste-Thistlethwaite's table, $8^4_{1}$ is also the link $L8a21$.

\begin{figure}[htbp]
	\begin{minipage}[t]{0.4 \linewidth}
		\centering
		\includegraphics[height=4.1cm,width=4.1cm]{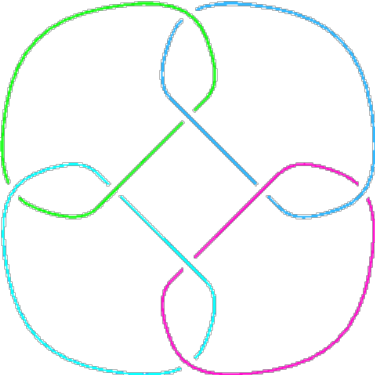}
	\end{minipage}
	\begin{minipage}[t]{0.5 \linewidth}
		\centering
		\includegraphics[height=4.1cm,width=4.1cm]{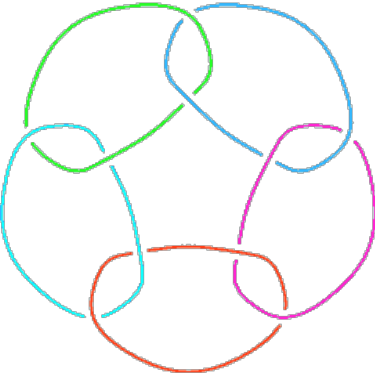}
	\end{minipage}
	\caption{ The chain links $8^4_1$ (the left) and $10^5_1$ (the right).  Link diagrams from https://knotinfo.math.indiana.edu/}
	\label{fig:841}
\end{figure}

Let $A=I_1I_2 \in \Delta_{4,\infty,\infty;\infty}$ be a parabolic element, which stabilizes the infinity. 
We will prove Theorem \ref{thm:4pp} by using the Ford domain $D_{\Sigma}$ of the even subgroup $\Sigma$ of $\Delta_{4,\infty,\infty;\infty}$. But the topology of the boundary at infinity of $D_{\Sigma}$, denoted by $\partial_{\infty}D_{\Sigma}$, is a  more complicated 3-manifold  than those in \cite{Deraux:2016, ParkerWill:2017, jwx, MaXie2020}. 
One reason is that the boundary of $\partial_{\infty}D_{\Sigma}$ is an infinite genus
(singular) surface, but all those in \cite{Deraux:2016, ParkerWill:2017, jwx, MaXie2020} are unbounded cylinders. 
Moreover, even though we know the boundary of $ \partial_{\infty}D_{\Sigma}$  is an infinite genus surface, 
it is also highly non-trivial to prove that $ \partial_{\infty}D_{\Sigma}$ is an infinite genus
handlebody, rather than a more complicated 3-manifold  with incompressible boundary.


Figure \ref{figure:4pppford} indicates that $\partial_{\infty}D_{\Sigma}$ should be an infinite genus
handlebody, where each "hole" corresponds to a genus of the handlebody $\partial_{\infty}D_{\Sigma}$.  
To prove that $\partial_{\infty}D_{\Sigma}$ is a handlebody rigorously, our idea is to introduce a set of $A$-invariant disks $A^k(E_B \cup E_{B^{-1}})$ for $k \in \mathbb{Z}$, which cap off the "holes" in Figure \ref{figure:4pppford}.
 Figure 	\ref{figure:4pppfordabstract} is  a combinatorial model of Figure \ref{figure:4pppford}. Figure  \ref{figure:4pppfordtop} is a small part of Figure 	\ref{figure:4pppfordabstract}.
The disks $E_B$ and $E_{B^{-1}}$ are foliated by affine-arcs, so they are ruled surfaces. 
The intersections  $E_B\cap\partial_{\infty}D_{\Sigma}$ and $E_{B^{-1}}\cap\partial_{\infty}D_{\Sigma}$ are exactly $\partial E_B$ and $\partial E_{B^{-1}}$, respectively:
\begin{itemize}

\item $E_B$ is a purple topological quadrilateral
in Figure \ref{figure:4pppfordtop};

\item $E_{B^{-1}}$ is a red topological quadrilateral in Figure \ref{figure:4pppfordtop}.
	
	
\end{itemize}
Figure \ref{figure:E1E2} shows a realistic view of $E_B$ and $E_{B^{-1}}$. 
Then the set of disks $A^k(E_B \cup E_{B^{-1}})$, for $k \in \mathbb{Z}$, cuts the 3-manifold $\partial_{\infty}D_{\Sigma}$ into a horotube. In other words, $$\partial_{\infty}D_{\Sigma}-(\cup_{k \in \mathbb{Z}}A^k(E_B \cup E_{B^{-1}}))$$ is an  $A$-invariant $(\mathbb{R}^2-\mathbb{D}^2) \times\mathbb{R}$ topologically. 
Taking the fundamental domain of $A$-action on $\partial_{\infty}D_{\Sigma}-(\cup_{k \in \mathbb{Z}} A^k(E_B \cup E_{B^{-1}}))$, using the information from side-pairings of the Ford domain, we can calculate the fundamental group of the 3-manifold $M$ at infinity of $\Sigma$. Thus, we get Theorem \ref{thm:4pp}.

Since there are too many conjugate classes of  parabolic elements  in the group  $\Sigma$, 
any fundamental domain of $\Sigma$ can not be very simple. 
For example, if we take a Dirichlet domain $D$ of  $\Sigma$ centred at the fixed point of $I_2I_3$, then the boundary at infinity $\partial_{\infty}D$ seems to be a genus three topological handlebody. 
It is very difficult to prove that it is exactly a genus three handlebody, compared with the argument in \cite{der-fal}.  
For a Dirichlet domain $D$ of $\Delta_{3,3,4;\infty}$ in \cite{der-fal}, a key step is to show that $\partial_{\infty}D$ is a solid torus, that is, a genus one handlebody!  
Even though our proof of Theorem \ref{thm:4pp}  is much involved,
it seems to the authors that the proof is the most elegant/effective approach.

The fact that there are four, rather than three, cusps of the 3-manifold $M$ at infinity of the even subgroup of $\Delta_{4, \infty,\infty;\infty}$ is unexpected. 
After we have convinced ourselves of this fact (which is not difficult), we continue to study the 3-manifold $M_p$ at infinity of the even subgroup of the complex hyperbolic triangle group $\Delta_{p, \infty,\infty;\infty}$ for $5\leq p \leq 9$.
Inspired by Theorem \ref{thm:4pp} and the main result in \cite{MaXie2021},  we conjecture that the 3-manifold $M_p$  is the complement of the chain link $C_p$ in the 3-sphere. 
Moreover, supported by previous results of various authors, we conjecture that the chain link $C_p$ is an ancestor of 3-manifolds at infinities of complex hyperbolic triangle groups $\Delta_{p,q,r;\infty}$ for each $3 \leq p \leq 9$, see Section \ref{sec:3mfdconj} for more details. 
In other words, the illuminating picture of the 3-manifold at infinity of every type $A$ critical complex hyperbolic triangle group begins to emerge. 
 
 
\textbf{Outline of the paper}: In Section \ref{sec:background}  we give the well-known background
material on the complex hyperbolic plane.  
In Section \ref{section:ford}, we give the combinatorial information of the Ford domain for the group $\Delta_{4,\infty, \infty;\infty}$. 
In Section \ref{section:topology4pppford}, we study the topology of the ideal boundary of the Ford domain of $\Delta_{4,\infty,\infty;\infty}$.
Based on the results in Sections \ref{section:ford} and \ref{section:topology4pppford}, we show that the 3-manifold at infinity of the even subgroup of $\Delta_{4,\infty,\infty;\infty}$ is the hyperbolic chain link $8_1^4$ in Section \ref{section:3mfd4ppp}. 
Finally,
we discuss related topics on 3-manifolds associated with the complex hyperbolic triangle groups with accidental parabolic elements in
Section \ref{sec:3mfdconj}.

\textbf{Acknowledgement}: The authors would like to thank Nathan Dunfield for the helpful communication on SnapPy via email. 
We would also like to thank Fangting Zheng for the help on Magma.
\section{Background}\label{sec:background}
The complex hyperbolic plane is a simply connected K\"{a}hler manifold with constant holomorphic sectional curvature $-1$.
This section provides some basic knowledge of the complex hyperbolic plane. 
We also refer the reader to Goldman's book \cite{Go} for more details.

\subsection{The complex hyperbolic plane}
Consider the Hermitian matrix
$$
H=\left[
\begin{array}{ccc}
0 & 0 & 1 \\
0 & 1 & 0 \\
1 & 0 & 0 \\
\end{array}
\right].
$$
The Hermitian form on ${\mathbb{C}}^3$ associated with $H$ is given by
$\langle {\bf{z}}, {\bf{w}} \rangle={\bf{w}^{*}}H{\bf{z}}$, which has signature $(2,1)$. 
Then ${\mathbb{C}}^3\backslash\{0\}$ is the union of negative cone $V_{-}$, null cone $V_{0}$ and positive cone $V_{+}$, where
\begin{eqnarray*}
	V_{-} &=& \left\{ {\bf{z}}\in {\mathbb{C}}^3: \langle {\bf{z}}, {\bf{z}} \rangle <0 \right\}, \\
	V_{0} &=& \left\{ {\bf{z}}\in {\mathbb{C}}^3: \langle {\bf{z}}, {\bf{z}} \rangle =0 \right\}\backslash\{0\} , \\
	V_{+} &=& \left\{ {\bf{z}}\in {\mathbb{C}}^3: \langle {\bf{z}}, {\bf{z}} \rangle >0 \right\}.
\end{eqnarray*}

\begin{defn}
	Let $P: \mathbb{C}^3\backslash\{0\} \longrightarrow \mathbb{C}P^2$ be the natural projectivization.
	The \emph{complex hyperbolic plane} $\hc$ is defined to be $P(V_{-})$, and its {boundary} $\partial\hc$ is $P(V_{0})$.
	The \emph{Bergman metric} on $\hc$ is given by the distance formula
	\begin{equation}  \label{eq:bergman-metric}
	\cosh^2\left(\frac{d(u,v)}{2}\right)=\frac{\langle {\bf{u}}, {\bf{v}} \rangle\langle {\bf{v}}, {\bf{u}} \rangle}{\langle {\bf{u}}, {\bf{u}} \rangle \langle {\bf{v}}, {\bf{v}} \rangle},
	\end{equation}
	where ${\bf{u}}, {\bf{v}} \in {\mathbb{C}}^3$ are lifts of $u,v$. Here $d(u,v)$ is the \emph{Bergman distance} between two points $u,v \in \hc$.
\end{defn}
The projection model of $\hc$ defined above is also called the {\it Siegel domain}.
Let $q_{\infty}$ be the point at infinity, and let $(z_1,z_2)^T$ be a point of $\overline {\hc}=\hc \cup \partial \hc$.
We use $\mathbf{q}_{\infty}=(1,0,0)^T$ to denote the lift of $q_{\infty}$,
and $(z_1,z_2,1)^T$ as the standard lift of $(z_1,z_2)^T$.

Let $\mathbf{U}(2,1)=\langle A\in \mathbf{GL}(3,\mathbb{C})|A^{*}HA=H\rangle$, 
and $\mathbf{PU}(2,1)=\mathbf{U}(2,1)/\mathbf{U}(1)$. 
Let $\omega$ be the cubic root.
Then $\mathbf{PU}(2,1)=\mathbf{SU}(2,1)/\langle I,\omega I,\omega^2I\rangle$,
which is the holomorphic isometry group of ${\bf H}^2_{\mathbb C}$. 

Elements of $\mathbf{SU}(2,1)$ fall into three types.
An isometry is {\it loxodromic} if it has exactly two fixed points on $\partial \hc$ 
and {\it parabolic} if it has exactly one fixed point on $\partial \hc$. 
An isometry is called {\it elliptic} if it has at least one fixed point in ${\bf H}^2_{\mathbb C}$.
An elliptic element is called {\it regular elliptic} if it has three distinct eigenvalues, otherwise, it is called {\it special elliptic}.

The traces of their matrix realizations can determine the types of isometries; see Theorem 6.2.4 in Goldman \cite{Go}. 
Assume $A\in{\rm{\mathbf{SU}}}(2,1)$ is nontrivial with  real trace. Then $A$ is elliptic if $-1\leq{\rm{tr}(A)}<3$.
In particular, if ${\rm{tr}(A)}=-1,0,1$, $A$ is elliptic of order 2, 3 and 4, respectively.
Moreover, $A$ is {\it unipotent} if ${\rm{tr}(A)}=3$. 

Let $\iota$ be given on the level of homogeneous coordinates by complex conjugation $(z_1,z_2,z_3)^T\longmapsto(\overline{z_1},\overline{z_2},\overline{z_3})^T$, which is an anti-holomorphic isometry of $\hc$.
Any other anti-holomorphic isometry may be written as the composition of $\iota$ and an element of $\mathbf{U}(2,1)$.

\subsection{Totally geodesic submanifolds and complex reflections}
There are five kinds of totally geodesic submanifolds 
in ${\hc}$: points, real geodesics, {\it complex lines} (copies of ${\bf H}^1_{\mathbb C}$), {\it Lagrangian planes} (copies of ${\bf H}^2_{\mathbb R}$) and $\hc$. 
Complex lines are the intersections of projective lines in $\mathbb{C}P^2$ with $\hc$, which are also called complex geodesics.
Lagrangian planes are the images of ${\bf H}^2_{\mathbb R}$ under the action of $\mathbf{PU}(2,1)$.
As the Riemannian sectional curvature of $\hc$ is non-constant, there is no totally geodesic hypersurface in $\hc$.

The ideal boundary of a complex line or a Lagrangian plane on $\partial\hc$ is called a $\mathbb{C}$-circle or a $\mathbb{R}$-circle respectively.
Let $C$ be a complex line.  
A {\it polar vector} of $C$ is the unique positive vector up to scalar multiplication, perpendicular to $C$ with
respect to the Hermitian form. 
Each positive vector corresponds to a complex line.

There is a special class of elliptic elements of order two in $\mathbf{PU}(2,1)$ related to complex lines.
\begin{defn}
	The \emph{complex reflection} on a complex line $C$ with polar vector ${\bf{n}}$ is a holomorphic isometry of order 2, fixing $C$ pointwise, defined by:
	\begin{equation}\label{eq:involution}
	I_{C}({\bf{z}})=-{\bf{z}}+\frac{2\langle {\bf{z}}, {\bf{n}} \rangle}{\langle {\bf{n}}, {\bf{n}} \rangle} {\bf{n}}.
	\end{equation}
	The complex line $C$ is usually called the {\it mirror} of $I_{C}$.
\end{defn}

\subsection{The Heisenberg group}
Let $\mathcal{N}=\mathbb{C}\times \mathbb{R}$ be the \emph{Heisenberg group} with product
\begin{equation*}
    (z,t)\cdot(\zeta,\nu)=(z+\zeta,t+\nu+2\Im(z\bar{\zeta})).
\end{equation*}
Then $\overline{\hc}$ can be identified with ${\mathcal{N}}\times{\mathbb{R}_{\geq0}}\cup \{q_{\infty}\}$.
Any point $q=(z,t,u)\in{\mathcal{N}}\times{\mathbb{R}_{\geq0}}$ has the standard lift
$$
{\bf{q}}=\begin{bmatrix}
\frac{-|z|^2-u+it}{2} \\
z \\
1 \\
\end{bmatrix}.
$$
Here $(z,t,u)$ is called the \emph{horospherical coordinates} of $\overline {\hc}$. 

When $u=0$, the points are on $\partial\hc$.
Hence, we directly identify
$\partial\hc$ with the union $\mathcal{N}\cup \{q_{\infty}\}$. 
Fixed $u_0 >0$, the level set $u=u_0$ is called the \emph{horosphere}  based at $q_{\infty}$, and the sup-level set $u\ge u_0$ is called the \emph{horoball} based at $q_{\infty}$.

The full stabilizer of $\mathbf{q}_{\infty}$ is generated by the isometries of the forms
\begin{equation*}
T_{[z,t]}=\begin{bmatrix}
1 & -\overline{z}& \frac{-|z|^{2}+it}{2} \\ 0 & 1 & z \\ 0 & 0 & 1 \end{bmatrix},\ 
R_{\theta}=\begin{bmatrix}
1 & 0& 0 \\ 0 & e^{i\theta} & 0 \\ 0 & 0 & 1 \end{bmatrix} \ 
\text{and}   \
D_{\lambda}=\begin{bmatrix}
\lambda & 0 & 0 \\ 0 & 1 & 0 \\ 0 & 0 & 1/\lambda \end{bmatrix},
\end{equation*}
where $[z,t]\in\mathcal{N}$, $\theta,\lambda\in \mathbb{R}$ and $\lambda \neq 0$.  
They are called {\it Heisenberg translation}, {\it Heisenberg rotation} and {\it Heisenberg dilation}, respectively.
Their actions on $\mathcal{N}$ are as follows:
\begin{equation*}
    \begin{aligned}
        T_{(z,t)}: (\zeta,\nu)&\longmapsto (z+\zeta,t+\nu+2{\rm{Im}}(z\bar{\zeta}));\\
        R_{\theta}: (\zeta,\nu)&\longmapsto (e^{i\theta}\zeta,\nu);\\
        D_{\lambda}: (\zeta,\nu)&\longmapsto(\lambda \zeta,\lambda^2 \nu).
    \end{aligned}
\end{equation*}

The \emph{Cygan metric} on $\mathcal{N}$  for $p=(z,t)$ and $q=(w,s)\in\mathcal{N}$ is defined by:
	\begin{equation}\label{eq:cygan-metric}
	d_{\textrm{Cyg}}(p,q)=|2\langle {\bf{p}}, {\bf{q}} \rangle|^{1/2}=\left| |z-w|^2-i(t-s+2\Im(z\bar{w})) \right|^{1/2}.
	\end{equation}
In fact, it is the restriction to $\mathcal{N}$ of the \emph{extended Cygan metric}, defined for $p=(z,t,u)$ and $q=(w,s,v)\in\mathcal{N}\times{\mathbb{R}_{\geq 0}}$ by the formula
\begin{equation}\label{eq:cygan-metric-extend}
	d_{\textrm{Cyg}}(p,q)=\left| |z-w|^2+|u-v|-i(t-s+2\Im(z\bar{w})) \right|^{1/2}.
	\end{equation}
The \emph{Cygan sphere} with center $(z_0,t_0)\in\mathcal{N}$ and radius $r>0$ in $\overline {\hc}$, is defined by
\begin{equation}\label{eq:cygan-sphere}
  S_{[z_0,t_0]}(r)=\{ (z,t,u)\in\overline {\hc}: \big||z-z_0|^2+u+\rm{i}(t-t_0+2\Im( z\overline{z_0}))\big|=r^2 \}.  
\end{equation}
Note that $S_{[z_0,t_0]}(r)$ can be obtained by the action of $T_{[z_0,t_0]}$ on $S_{[0,0]}(r)$.
There is another way to define a Cygan sphere:
\begin{defn}\label{def:geographic}
	The \emph{geographic coordinates} $(\alpha,\beta,w)$ of $q=q(\alpha,\beta,w)\in S_{[0,0]}(r)$ is given by the lift
	\begin{equation}\label{eq:geog-coor}
	{\bf{q}}={\bf{q}}(\alpha,\beta,w)=\left[
	\begin{array}{c}
	-r^2e^{-i\alpha}/2 \\
	rwe^{i(-\alpha/2+\beta)} \\
	1 \\
	\end{array}
	\right],
	\end{equation}
	where $\alpha\in [-\pi/2,\pi/2]$, $\beta\in [0, \pi)$ and $w\in [-\sqrt{\cos(\alpha)},\sqrt{\cos(\alpha)}]$.
	In particular, the ideal boundary of $S_{[0,0]}(r)$ on $\partial\hc$ are the points with $w=\pm\sqrt{\cos(\alpha)}$.
\end{defn}

The following property describes the characteristics of the level sets for $\alpha$ and $\beta$.
\begin{prop}[Goldman \cite{Go}, Parker and Will \cite{ParkerWill:2017}]
Let $S_{[0,0]}(r)$ be a Cygan sphere with geographic coordinates $(\alpha,\beta,w)$.
\begin{enumerate}[(i)]
	\item For each $\alpha_0\in [-\pi/2,\pi/2]$, the set of points in $S_{[0,0]}(r)$ with $\alpha=\alpha_0$  is a complex line, called a slice of $S_{[0,0]}(r)$.
	\item For each $\beta_0\in [0,\pi)$, the set of points in $S_{[0,0]}(r)$ with $\beta=\beta_0$ is a Lagrangian plane, called a meridian of $S_{[0,0]}(r)$. 
	\item The set of points with $w=0$ is called the real spine of $S_{[0,0]}(r)$.
	\end{enumerate}
\end{prop}

About the intersection of Cygan spheres, 
Goldman, Parker and Will give us the following result.
\begin{prop}[Goldman \cite{Go}, Parker and Will \cite{ParkerWill:2017}] \label{prop:Giraud}
The intersection of two Cygan spheres is a smooth disk.
\end{prop}

\begin{rem}
   This disk is frequently called a {\it Giraud disk}.
\end{rem}

\subsection{Isometric spheres and Ford polyhedron}
Let $g=(g_{ij})$ be an element of $\mathbf{PU}(2,1)$ not fixing $\mathbf{q}_{\infty}$. 
This implies that $g_{31}\neq 0$.
\begin{defn} \label{def:isosphere}
	The \emph{isometric sphere} of $g$, denoted by $ I(g)$, is the set
	\begin{equation}\label{eq:isom-sphere}
	I(g)=\{ p \in\overline{\hc}: |\langle {\bf{p}}, {\bf{q}}_{\infty} \rangle | = |\langle {\bf{p}}, g^{-1}({\bf{q}}_{\infty}) \rangle| \}.
	\end{equation}
\end{defn}
An isometric sphere is a 3-ball.
The \emph{spinal sphere} of $I(g)$ or $g$ is defined to be $I(g)\cap\partial\hc$, denoted by $\partial_{\infty}I(g)$, which is a 2-sphere.
Besides, $I(g)$ is exactly the Cygan sphere with center
$g^{-1}({\bf{q}}_{\infty})\in\mathcal{N}$
and radius $r_g=\sqrt{2/|g_{31}|}$.
The \emph{interior} of $I(g)$ is the set
\begin{equation}\label{eq:exterior}
\{ p \in \overline{\hc} : |\langle {\bf{p}}, {\bf{q}}_{\infty} \rangle | > |\langle {\bf{p}}, g^{-1}({\bf{q}}_{\infty}) \rangle| \}.
\end{equation}
The \emph{exterior} of $I(g)$ is the set
\begin{equation}\label{eq:interior}
\{ p \in \overline{\hc}: |\langle {\bf{p}}, {\bf{q}}_{\infty} \rangle | < |\langle {\bf{p}}, g^{-1}({\bf{q}}_{\infty}) \rangle| \}.
\end{equation}

We summarize some properties on isometric spheres in the following proposition.
\begin{prop}[\cite{Go}, Section 5.4.5] 
Let $g$ and $h$ be elements of $\mathbf{PU}(2,1)$ which do not fix $\mathbf{q}_{\infty}$ and $g^{-1}(\mathbf{q}_{\infty})\neq h^{-1}(\mathbf{q}_{\infty})$.
Let $f \in\mathbf{PU}(2,1)$ be a unipotent map fixing $\mathbf{q}_{\infty}$.  
Then the following hold:
\begin{enumerate}[(i)]
		\item  $g(I(g))=I(g^{-1})$, and $g$ maps the exterior of $I(g)$ to the interior of $I(g^{-1})$.
		\item  $I(gf)=f^{-1}I(g)$, $I(fg)=I(g)$.
		\item $g(I(g)\cap I(h))=I(g^{-1})\cup I(hg^{-1})$, $h(I(g)\cap I(h))=I(gh^{-1})\cup I(h^{-1})$.
	\end{enumerate}
\end{prop}

\begin{defn}The \emph{Ford domain} $D_{\Gamma}$ for a discrete group $\Gamma < \mathbf{PU}(2,1)$ centred at $q_{\infty}$ is the intersection of the closures of the exteriors of all isometric spheres of elements in $\Gamma$ not fixing $\mathbf{q}_{\infty}$. 
 That is,
	$$D_{\Gamma}=\{p\in \overline{\hc}: |\langle \mathbf{p},\mathbf{q}_{\infty}\rangle|\leq|\langle \mathbf{p},g^{-1}(\mathbf{q}_{\infty})\rangle|
	\ \forall g\in \Gamma \ \mbox{with} \ g(\mathbf{q}_{\infty})\neq \mathbf{q}_{\infty} \}.$$
\end{defn}

The boundary at infinity of the Ford domain is made up of pieces of isometric spheres. 
When $q_{\infty}$ is in the domain of discontinuity or is a parabolic fixed point, the Ford domain is preserved by $\Gamma_{\infty}$, the stabilizer of
$q_{\infty}$ in $\Gamma$.  
In this case, 
the fundamental domain for
$\Gamma$ is the intersection of $D_{\Gamma}$ with a fundamental domain for $\Gamma_{\infty}$. 
Facets of codimensions one, two, three and four in $D_{\Gamma}$ are called {\it sides}, {\it ridges}, {\it edges} and {\it vertices}, respectively. 
Moreover, a {\it bounded ridge} is a ridge that does not intersect $\partial {\bf H}^2_{\mathbb C}$.
Otherwise, the ridge is called an {\it infinite ridge}.

It is usually very hard to determine the Ford domain $D_{\Gamma}$ because one should check infinitely many inequalities.
A general method is to guess the Ford polyhedron and check it by using the Poincar\'e polyhedron theorem.
The sides of $D_{\Gamma}$ should be paired by isometries, and the images of $D_{\Gamma}$
under these so-called side-pairing maps should give a local tiling of ${\bf H}^2_{\mathbb C}$.  
If they do, and if, moreover, the quotient of $D_{\Gamma}$ by the identification given by the side-pairing maps is complete, the Poincar\'{e} polyhedron
theorem implies that the images of $D_{\Gamma}$ give a global tiling of $\hc$.

Once a fundamental domain is obtained, one gets a presentation of $\Gamma$. 
Its generators are given by the side-pairing maps together with a generating set of $\Gamma_{\infty}$. 
Its relations are ridge cycles, which correspond to the local tilings, bearing each codimension-two face. 
For more details on the Poincar\'e polyhedron theorem, see \cite{dpp, ParkerWill:2017}.

\section{Combinatoric of Ford domain of %
\texorpdfstring{$\Delta_{4,\infty,\infty;\infty}$}{the complex hyperbolic group} %
}
\label{section:ford}

In this section, we study the combinatoric of the Ford domain of $\Delta_{4,\infty,\infty;\infty}$. The main result is Theorem \ref{thm:fundamental-domain4ppp}.

\subsection{The matrices representation of the triangle group %
\texorpdfstring{$\Delta_{4,\infty,\infty;\infty}$}{the complex hyperbolic group} %
}
\label{subsection:4ppp}

Suppose that two complex reflections $J_1$ and $J_2$ are
in $\mathbf{SU}(2,1)$ such that $J_1J_2$ is a unipotent map fixing $q_{\infty}$. 
Conjugating by a Heisenberg rotation and a Heisenberg dilation if necessary, we may suppose that $J_1$ and $J_2$ are
\begin{equation}\label{eq-J1-J2}
J_1=\left[\begin{matrix}
-1 & 0& 0 \\ 0 & 1 &  0\\ 0 & 0 & -1 \end{matrix}\right], \quad
J_2=\left[\begin{matrix}
-1 & -2 & 2 \\ 0 & 1 & -2 \\ 0 & 0 & -1 \end{matrix}\right].
\end{equation}
We want to find $J_3$ such that $J_2J_3$ is an elliptic element of order 4. 
Besides, $J_3$ satisfies that $J_3J_1$, $J_1J_2$, and $J_1J_3J_2J_3$ are all unipotent. 
Assume that the polar vector of the mirror of $J_3$  is 
\begin{eqnarray*}
	{\bf n}_3 = \left[\begin{matrix} x \\ y+z \cdot  {\rm i} \\ 1\end{matrix}\right],
\end{eqnarray*}
where $x,y,z \in \mathbb{R}$. Then 
$J_3$ has the following form
$$J_3=\left[\begin{matrix}
\frac{-y^2-z^2}{y^2+z^2+2x}& \frac{2x(y- z {\rm i})}{y^2+z^2+2x}& \frac{2x^2}{y^2+z^2+2x} \\ \frac{2(y+z {\rm i})}{y^2+z^2+2x} &\frac{y^2+z^2-2x}{y^2+z^2+2x}&\frac{2x(y+ z {\rm i})}{y^2+z^2+2x} \\ \frac{2}{y^2+z^2+2x} & \frac{2(y- z {\rm i})}{y^2+z^2+2x} & \frac{-y^2-z^2}{y^2+z^2+2x}\end{matrix}\right].$$
Note that $J_3J_1$ is unipotent, which implies that the trace of $J_3J_1$ is $3$ and then $x=0$. 
Similarly, since $J_1J_2$ and  $J_1J_3J_2J_3$ are unipotent, we have $y=1$ and $z= \pm 1$. The groups corresponding to $z= \pm 1$ are conjugated by an anti-holomorphic isometry.   
In this paper, we take $z=1$, then 
$$J_3=\left[\begin{matrix}
-1& 0& 0 \\ 1+ {\rm i}& 1 &0 \\1 & 1-{\rm i} & -1 \end{matrix}\right].$$
Moreover,  $\langle J_1, J_2,J_3 \rangle $  is  a subgroup of the Gauss-Picard  modular group $\mathbf{PU}(2,1;\mathbb{Z}[{ \rm i}])$, proving that $\langle J_1,J_2,J_3\rangle $ is discrete. 	
The discreteness of $\Delta_{4,\infty,\infty;\infty}$ was first noted by Thompson \cite{Thompson:2010}. 
 
For technical reasons, we conjugate the matrices representation of $\Delta_{4,\infty,\infty;\infty}$ by 
$$T=\left[\begin{matrix}
 1& 1& \frac{-1-2 {\rm i}}{2} \\ 0 &1&-1 \\ 0 & 0 & 1\end{matrix}\right].$$
That is, let $I_1=T J_1T^{-1}$, $I_2=T J_2T^{-1}$ and $I_3=T J_3T^{-1}$ be generators of $\Delta_{4,\infty,\infty;\infty}$. 
Specifically: 
\begin{equation}\label{eq-I1-I2_I3}
 I_1=\left[\begin{matrix}
 -1 & 2& 2 \\ 0 & 1 &  2\\ 0 & 0 & -1 \end{matrix}\right], \quad
 I_2=\left[\begin{matrix}
 -1 & 0 & 0 \\ 0 & 1 & 0 \\ 0 & 0 & -1 \end{matrix}\right],\quad
 I_3=\left[\begin{matrix}
 -1/2 & -{\rm i}/2& 1/4 \\ {\rm i} & 0 & {\rm i}/2 \\ 1 & -{\rm i} & -1/2\end{matrix}\right].
 \end{equation}

\subsection{Ford domain of %
\texorpdfstring{$\Delta_{4,\infty,\infty;\infty}$}{the complex hyperbolic group} %
}
\label{subsection:ford}

Let $\Gamma=\langle I_1, I_2, I_3\rangle$ be the group $\Delta_{4,\infty,\infty;\infty}$.
Let $A=I_1I_2$, $B=I_2I_3$. 
The subgroup generated by A and B, denoted by 
$\Sigma=\langle A, B\rangle$, is the even subgroup of $\Gamma$ with index two. 
By direct calculation, we have
\begin{equation*}\label{eq-A-B}
A=\left[\begin{matrix}
1 & 2& -2 \\ 0 & 1 & -2 \\ 0 & 0 & 1 \end{matrix}\right], \quad
B=\left[\begin{matrix}
\frac{1}{2} &\frac{i}{2}&-\frac{1}{4} \\i & 0&\frac{i}{2}\\ -1 &i &\frac{1}{2} \end{matrix}\right].
\end{equation*}
Due to the definition of $\Delta_{4,\infty,\infty;\infty}$, we have that $B$ is elliptic of order 4, and there are three unipotent elements, $A$, $AB$ and $AB^2$. 
Besides, $AB^{-1}$ is also unipotent. 
For future reference, we provide the lifts of their fixed points, as vectors in $\mathbb{C}^{3}$:
\begin{equation*}\label{eq-x1x2x3}
\mathbf{p}_{AB}=x_1=\left[\begin{matrix}
-\frac{1+2{\rm i}}{2} \\ -1  \\  1 \end{matrix}\right], \ 
\mathbf{ p}_{AB^2}=x_2=\left[\begin{matrix}
-\frac{1}{2} \\ -1 \\ 1  \end{matrix}\right] \  \text{and} \ 	\mathbf{ p}_{AB^{-1}}=x_3=\left[\begin{matrix}
\frac{-1+2{\rm i}}{2}\\-1\\ 1\end{matrix}\right]
\end{equation*}

In this subsection, we will construct a Ford domain $D_{\Sigma}$ whose sides are contained in the isometric spheres of some elements of $\Sigma$.  
The reader is preferred to Figures 	\ref{figure:4pppford} and \ref{figure:4pppfordabstract} for the views of 
$D_{\Sigma}$. 
Figure \ref{figure:4pppford} is a realistic view of the ideal boundary of $D_{\Sigma}$ with center $q_{\infty}$, 
which is just for motivation. 
In this subsection, we will show that the local picture of this figure is correct.

In Figure \ref{figure:4pppford}:
\begin{itemize}[-]
	\item the green, black and red spheres are   the spinal spheres of $A^{-1}BA$, $B$ and $ABA^{-1}$ respectively;
	\item the blue, purple and steel-blue spheres are   the spinal spheres of $A^{-1}B^{-1}A$, $B^{-1}$ and $AB^{-1}A^{-1}$ respectively;
	\item the cyan, yellow and pink spheres are the spinal spheres of $A^{-1}B^{2}A$, $B^{2}$ and $AB^{2}A^{-1}$ respectively.
\end{itemize}
Figure \ref{figure:4pppfordabstract} is a combinatorial model of Figure \ref{figure:4pppford}.
Our computations later will show that the combinatorial model coincides with Figure \ref{figure:4pppford}. 

\begin{figure}[tbp]
	\begin{center}
		\begin{tikzpicture}
		\node at (0,0) {\includegraphics[width=11cm,height=6cm]{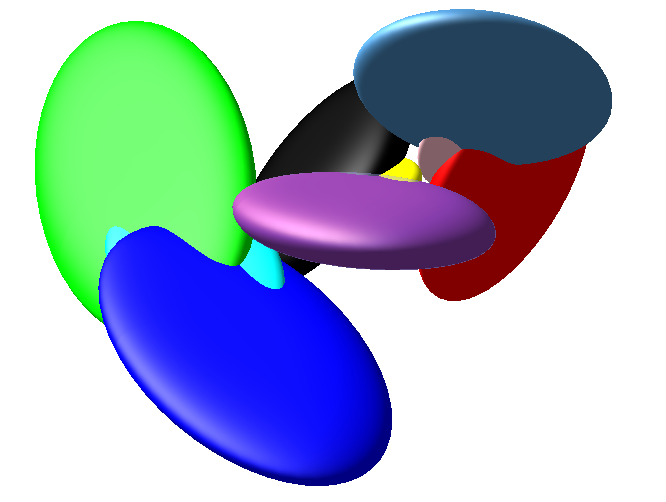}};
		\end{tikzpicture}
	\end{center}
	\caption{A  realistic view the ideal boundary of the Ford domain of $\Delta(4, \infty,\infty;\infty)$.}
	\label{figure:4pppford}
\end{figure}

Let  $I(B)$,  $I(B^{-1})$ and $I(B^2)$ be the isometric spheres for $B$,  $B^{-1}$ and $B^2$, respectively. 
We give some notations for the isometric spheres, which are the images of $I(B)$, $I(B^{-1})$ and $I(B^{2})$ by powers of $A$.
\begin{defn} For $k\in \mathbb{Z}$:
\begin{enumerate}
    \item  Let $I_{k}^{+}$ be the isometric sphere $I(A^{k}BA^{-k})=A^{k}I(B)$ and $\partial_{\infty}I_{k}^{+}$ be its corresponding spinal sphere.
    \item Let $I_{k}^{-}$ be the isometric sphere $I(A^{k}B^{-1}A^{-k})=A^{k}I(B^{-1})$ and $\partial_{\infty}I_{k}^{-}$ be its corresponding spinal sphere.
    \item Let $I^{\star}_{k}$ be the isometric sphere $I(A^{k}B^2A^{-k})=A^{k}I(B^2)$  and $\partial_{\infty}I_{k}^{*}$ be its corresponding spinal sphere. 
\end{enumerate}
\end{defn}
The action of $A$ on $\mathcal{N}$ is given by
\begin{equation}\label{eq-A}
(z,t)\longmapsto (z-2, t+4\Im(z)).
\end{equation}
In particular, we have the following.
\begin{prop}\label{prop:A-action}
For any value of $t_0\in \mathbb{R}$, 
$A$ preserves every $\mathbb{R}$-circle of the form $(x,t_0)$ with $x\in \mathbb{R}$.
\end{prop}

As $A$ is unipotent and fixes $q_{\infty}$, it is a Cygan isometry and preserves the radii of the isometric spheres.  
This directly proves 
\begin{prop}\label{prop:center-radius4ppp}
	For any integer $k\in \mathbb{Z}$:
    \begin{enumerate}[(i)]
	 \item The isometric spheres $I_{k}^{+}$ and $I_{k}^{-}$ have the same radius $\sqrt{2}$ and are centered at $(-2k+{\rm i},-4k)$ and $(-2k-{\rm i},4k)$ in Heisenberg coordinates, respectively.
	 \item The isometric sphere $I^{\star}_{k}$ has radius $1$ and is centered at $(-2k,0)$ in Heisenberg coordinates.
	\end {enumerate} 	
\end{prop}

We first consider the intersections of the isometric spheres $I_{k}^{\pm}$ and $I^{\star}_k$ for $k\in \mathbb{Z}$. 
By the $A$-symmetry, $I_{l}^{+} \cap I_{l+k}^{+}$ is non-empty  if and only if $I_{0}^{+} \cap I_{k}^{+}$ is non-empty. 
Similar properties hold for other intersections of pairs of isometric spheres.


\begin{prop}\label{prop:pair-disjoint4ppp}For any $k\in \mathbb{Z}$:
	\begin{enumerate}[(i)]
	\item  \label{item:Ikplus} $I_{0}^{+}$ is disjoint from $I_{k}^{+}$ when $|k| \geq 2$.
	\item  \label{item:Ikminus} $I_{0}^{-}$ is disjoint from $I_{k}^{-}$ when $|k| \geq 2$.
	\item  \label{item:IKplusminus} $I_{0}^{+}$ is disjoint from $I_{k}^{-}$ when $k \neq 0$.
	\item  \label{item:Jk}  $I^{\star}_{0}$ is disjoint from $I^{\star}_{k}$ when $|k| \geq 2$.
    \item  \label{item:JkIk} $I^{\star}_{0}$ is disjoint from $I^{\pm}_{k}$ when $|k| \geq 1$.
	\item  \label{item:I0plusI1minustangent}   $I_{0}^{+}$ is tangent with  $I_{1}^{-}$ at the point $x_1$.
	\item  \label{item:I0minusI1plustangent}   $I_{0}^{-}$  is tangent  with  $I_{1}^{+}$ at the point $x_3$.
	\item  \label{item:Jktangent}   $I^{\star}_{0}$ is  tangent  with  $I^{\star}_{1}$ at the point $x_2$.
	\end{enumerate}
\end{prop}
\begin{proof} 
For \ref{item:Ikplus}, the Cygan distance between the centers of  $I_{0}^{+}$ and $I_{k}^{+}$  is $|2k|$. 
It is greater than $2 \sqrt{2}$ when $|k| \geq 2$, so $I_{0}^{+}$ is disjoint from $I_{k}^{+}$ if $|k| \geq 2$.  
Similarly, we have \ref{item:Ikminus}.  

For \ref{item:IKplusminus}, the Cygan distance between the centers of $I_{0}^{+}$ and $I_{k}^{-}$ is $2 \sqrt{k^2+1}$, which is greater than $2 \sqrt{2}$ when $|k| \geq 2$.
Hence, $I_{0}^{+}$ is disjoint from $I_{k}^{-}$ if $|k| \geq 2$.  
When $k= \pm 1$, we assume that 
$$(x+y {\rm i},t) \in \partial_{\infty} I_{0}^{+} \cap \partial_{\infty} I_{ \pm 1}^{-}$$ 
in Heisenberg coordinates with $x,y,t \in \mathbb{R}$. 
Then considering the equation of an isometric sphere in Definition \ref{def:isosphere}, it is easy to show that there is no $(x+y {\rm i},t)$ lying in $\partial_{\infty} I_{0}^{+} \cap \partial_{\infty} I_{ \pm 1}^{-}$. 
This implies that $I_{0}^{+} \cap I_{ \pm 1}^{-}$ is empty in turn.

For \ref{item:Jk}, the Cygan distance between the centers of  $I^{\star}_{0}$ and $I^{\star}_{k}$  is $|2k|$. 
This is greater than $2$ when $|k| \geq 2$. So $I^{\star}_{0}$ is disjoint from $I^{\star}_{k}$ if $|k| \geq 2$. 
When $k= 1$, we assume  $(x+y {\rm i},s,v) \in I^{\star}_{0} \cap  I^{\star}_{ 1}$ in horospherical coordinates with $x,y,s \in \mathbb{R}$ and $v \in \mathbb{R}_{\ge0}$.  
The only solution is:
$$x=-1, \quad y=0, \quad s=0, \quad v=0. $$
This is exactly the point $\mathbf{p}_{AB^2}=x_2$.
When $k= -1$,  we assume  $(x+y {\rm i},s,v) \in  I^{\star}_0\cap  I^{\star}_{-1}$ in horospherical coordinates with $x,y,s \in \mathbb{R}$ and $v \in \mathbb{R}_{\ge0}$.
The only solution is   
$$x=1, \quad y=0, \quad s=0, \quad v=0. $$
This is the point $A^{-1}(x_2)$.  
In particular, we also proved \ref{item:Jktangent}.

For \ref{item:JkIk},  the Cygan distance between the centers of  $I^{\star}_{0}$ and  $I^{\pm }_{k}$  is $\sqrt{16k^4+24k^2+1}$, which is greater than $\sqrt{2}+1$ when $|k| \geq 1$. 
So $I^{\star}_{0}$ is disjoint from $I^{\pm}_{k}$ if $|k| \geq 1$. 

For \ref{item:I0plusI1minustangent}, we assume $(x+y {\rm i}, s,u) \in  I^{+}_{0} \cap  I^{-}_{1}$ in horospherical coordinates with $x,y,s \in \mathbb{R}$ and $u \in \mathbb{R}_{+}$. 
Consider the  equations of the isometric  spheres of  $I^{+}_{0}$ and $ I^{-}_{1}$ in Definition  \ref{def:isosphere}, we have 
\begin{flalign}
\left\{ \begin{aligned}
(x^2+(y-1)^2+u)^2+(s-2x)^2=4; \qquad& \  \\ 
((x+2)^2+(y+1)^2+u)^2+(s+4+2(x-2y))^2=4. \qquad& \  
\end{aligned}
\right.
\end{flalign}
It is easy to show that 
$$x=-1, \quad y=0, \quad s=-2, \quad u=0 $$ is the only solution of the above two equations, corresponding to point $\mathbf{p}_{AB}=x_1$.

For \ref{item:I0minusI1plustangent}, we assume
$(x+y {\rm i}, s,u) \in  I^{-}_{0} \cap  I^{+}_{1}$ in horospherical coordinates with $x,y,s \in \mathbb{R}$ and $u \in \mathbb{R}_{\ge0}$. 
Consider the  equations of the Cygan spheres of  $I^{+}_{0}$ and $ I^{-}_{1}$ in Definition \ref{def:isosphere}, we have 
\begin{flalign}
\left\{ \begin{aligned}
(x^2+(y+1)^2+u)^2+(s+2x)^2=4; \qquad& \  \\ 
\left((x+2)^2+(y-1)^2+u\right)^2+\left(s-4-2(x+2y)\right)^2=4. \qquad& \  
\end{aligned}
\right.
\end{flalign}
It is easy to show that $$ x = -1, \quad y = 0, \quad s = 2, \quad u = 0, $$ is the only solution of the above two equations, which is the point $\mathbf{p}_{AB^{-1}}=x_3$.
\end{proof}
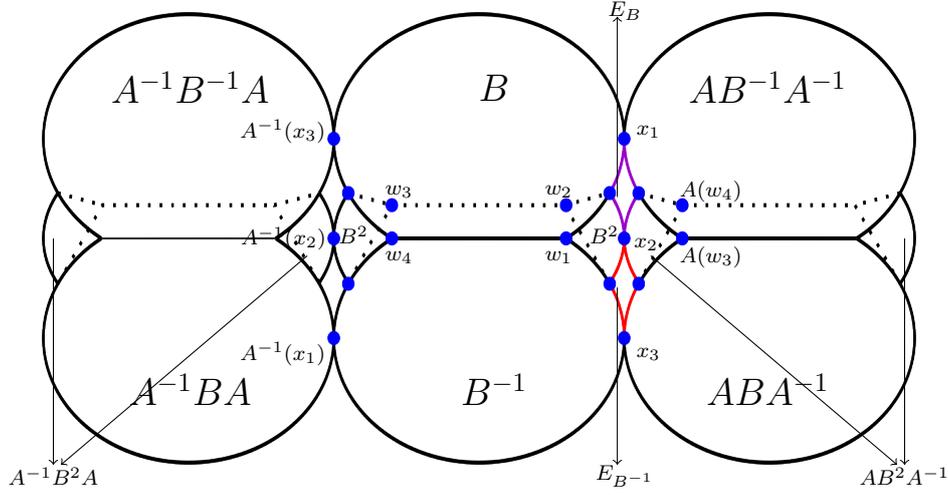
\begin{figure}[htbp]
	\begin{center}
		\begin{tikzpicture}
		\node at (0,0) {\includegraphics[width=11.5cm,height=6.0cm]{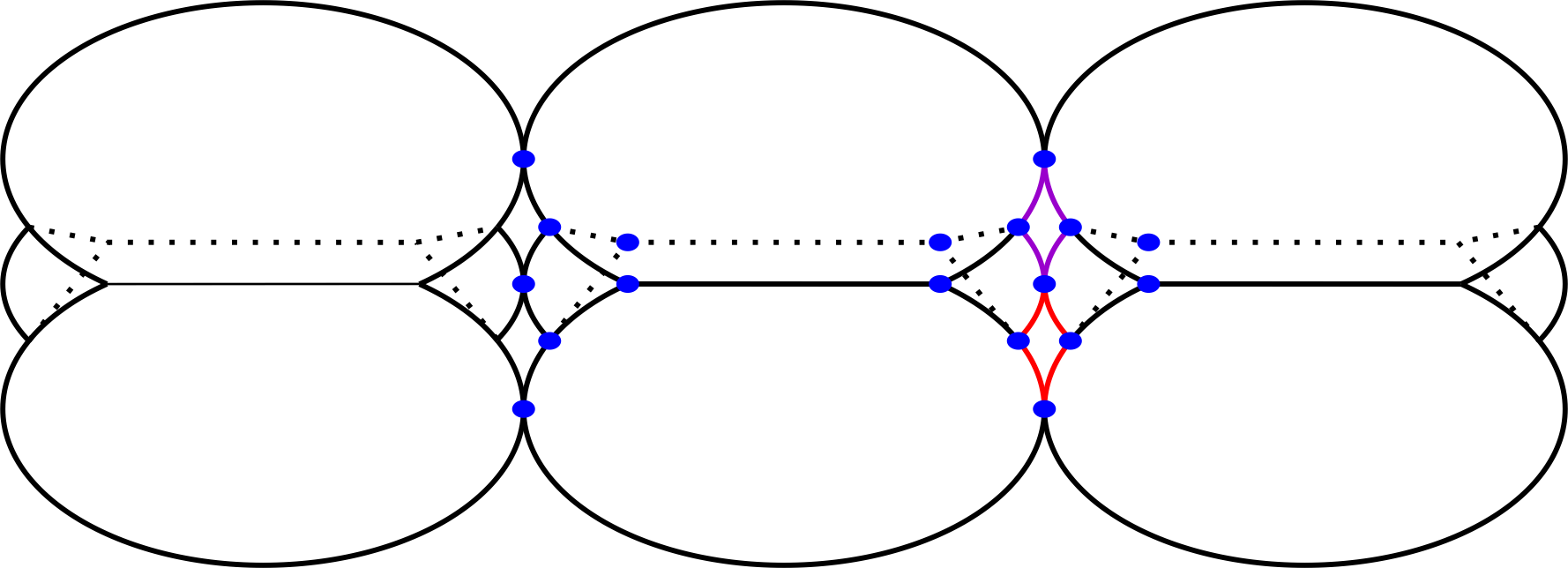}};

		\node at (0.2,2.0){\large $B$};
		\node at (0.2,-2.0){\large $B^{-1}$};
		
		\node at (3.8,-2.0){\large $ABA^{-1}$};
		\node at (3.8,2.0){\large $AB^{-1}A^{-1}$};
		
		\node at (-3.8,-2.0){\large $A^{-1}BA$};
		\node at (-3.8,2.0){\large $A^{-1}B^{-1}A$};

		\node at (2.23,1.42){\tiny $x_1$};
		\node at (2.2,-0.08){\tiny $x_2$};
		\node at (2.23,-1.53){\tiny $x_3$};
		
		\node at (-1.05,-0.25){\tiny $w_4$};
		\node at (-1.05,0.62){\tiny $w_3$};
		\node at (1.05,-0.25){\tiny $w_1$};
		\node at (1.05,0.62){\tiny $w_2$};
		
		\node at (-2.58,1.42){\tiny $A^{-1}(x_3)$};
		\node at (-2.57,-0.00){\tiny $A^{-1}(x_2)$};
		\node at (-2.58,-1.53){\tiny $A^{-1}(x_1)$};
		
		\node at (-1.65,0.05){\tiny $B^2$};
		\node at (1.65,0.05){\tiny $B^2$};

		\draw[->] (1.82,-0.65)--(1.82,-3);
		\node at (1.92,-3.15){\tiny $E_{B^{-1}}$};
        \draw[->] (1.82,0.55)--(1.82,2.94);
		\node at (1.92,3.02){\tiny $E_{B}$};

  \draw[->] (5.6,0)--(5.6,-3);
  \draw[->] (2.26,-0.23)--(5.5,-3);
		\node at (5.6,-3.15){\tiny $AB^2A^{-1}$};

  \draw[->] (-5.6,0)--(-5.6,-3);
  \draw[->] (-2.26,-0.23)--(-5.5,-3);
		\node at (-5.6,-3.15){\tiny $A^{-1}B^2A$};

		\node at (3.05,-0.25){\tiny $A(w_3)$};
		\node at (3.05,0.62){\tiny $A(w_4)$};
		\end{tikzpicture}
	\end{center}
	\caption{An abstract picture of the Ford domain of $\Sigma$.  For example, the sphere labeled by $B$ is the spinal sphere $\partial_{\infty}I(B)$, in the notation of Proposition \ref{prop:pair-disjoint4ppp}, it is $\partial_{\infty}I^{+}_0$. }
    \label{figure:4pppfordabstract}
\end{figure}


\begin{defn} The infinite polyhedron $D_{\Sigma}$ is the intersection of the exteriors of all the isometric spheres in $\{I_{k}^{+}, I_{k}^{-},I^{\star}_{k}: k\in \mathbb{Z}\}$. That is,
	$$D_{\Sigma}=\{p\in \hc: |\langle \mathbf{p}, \mathbf{q}_{\infty}\rangle| \leq |\langle \mathbf{p}, A^{k}B^i(\mathbf{q}_{\infty})\rangle|\text{, } \forall k\in \mathbb{Z} \ \text{and} \ i= 1, 2, 3\}.$$
\end{defn}

We shall use a version of the Poincar\'{e} polyhedron theorem for coset decompositions because $D_{\Sigma}$ is stabilized by the cyclic subgroup $\langle A\rangle$. 
Refer to \cite{dpp, ParkerWill:2017} for the precise statement of this version of the Poincar\'{e} polyhedron theorem. 
We proceed to show:

\begin{thm}\label{thm:fundamental-domain4ppp}
	The polyhedron $D_{\Sigma}$ is a fundamental domain for the cosets of $\langle A \rangle$ in $\Sigma=\langle A, B \rangle$. Moreover, the group $\Sigma$  is discrete and  has the presentation
	\begin{equation*}
	\langle  A, B: B^4=id\rangle.
	\end{equation*}
\end{thm}

We denote $I_{k}^{+} \cap D_{\Sigma}$ by $s_{k}^{+}$, denote $I_{k}^{-} \cap D_{\Sigma}$ by $s_{k}^{-}$ and denote $I^{\star}_{k} \cap D_{\Sigma}$ by $s^{\star}_{k}$ for any $k \in \mathbb{Z}$. 
We will show that they are $3$-faces of $D_{\Sigma}$.
Let
\begin{equation*}
w_1=\left[\begin{matrix}
{-\frac{1+2\sqrt{2}i}{6}} \\ 
{-\frac{\sqrt{2}+i}{3}}\\ 
1  \end{matrix}\right], 
w_2=\left[\begin{matrix}
{-\frac{1-2\sqrt{2}i}{6}} \\ 
{-\frac{\sqrt{2}-i}{3}}\\ 
1\end{matrix}\right],
w_3=\left[\begin{matrix}
{-\frac{1-2\sqrt{2}i}{6}} \\ 
{\frac{\sqrt{2}-i}{3}}\\ 
1  \end{matrix}\right]
\ \text{and}\
w_4=\left[\begin{matrix}
{-\frac{1+2\sqrt{2}i}{6}} \\ 
{\frac{\sqrt{2}+i}{3}} \\  
1 \end{matrix}\right]
\end{equation*}
be four points. 
Moreover, it can be checked directly 
$$B(w_1)=w_4, \quad B(w_4)=w_3,  \quad B(w_3)=w_2, \quad \text{and}\quad B(w_2)=w_1.$$

\begin{prop}\label{prop:4pppGirauddisk}
For the isometric spheres  $I_{0}^{+}$, $I_{0}^{-}$ and $I^{\star}_{0}$:
\begin{enumerate}[(i)]
\item Each of $I_{0}^{+} \cap I^{-}_{0}$,  $I_{0}^{+} \cap I^{\star}_{0}$ and  $I_{0}^{-} \cap I^{\star}_{0}$ is a Giraud disk.
\item The triple intersection $I_{0}^{+} \cap I^{-}_{0} \cap I^{\star}_{0}$ is a union of two geodesics crossed at the fixed point of $B$, whose four endpoints are on $ \partial {\bf H}^2_{\mathbb C}$.
\item The four points $w_1$, $w_2$, $w_3$ and $w_4$ are exactly the set $$I^{+}_0 \cap I^{-}_0 \cap I^{\star}_0 \cap \partial  {\bf H}^2_{\mathbb C}.$$
\end{enumerate}
\end{prop}

\begin{proof}
First note that the fixed point of $B$ in ${\bf H}^2_{\mathbb C}$ is 
\begin{equation*}\label{eq-fixB}
\mathbf{p}_{B}=\left[\begin{matrix}
-1/2 \\ 0 \\  1 \end{matrix}\right].
\end{equation*}
    It can be  checked directly that $\mathbf{p}_{B}$ lies in $I^{+}_0$, $I^{-}_0$ and $I^{\star}_0$. Moreover, $q_{\infty}$, $B(q_{\infty})$ and $B^2(q_{\infty})$ are linearly independent. By Proposition \ref{prop:Giraud}, $I_{0}^{-} \cap I^{\star}_{0}$ is a non-empty Giraud disk. 
Similarly, each of $I_{0}^{+} \cap I^{-}_{0}$ and $I_{0}^{+} \cap I^{\star}_{0}$ is a non-empty Giraud disk.
			
    We now consider the triple intersection $I_{0}^{+}\cap I^{-}_{0}\cap I^{\star}_{0}$. 
    From (\ref{eq:geog-coor}), the geographic coordinates $(\alpha,\beta,w)$ of points in $I^{\star}_0$
    is given by 
		\begin{equation}
		{\bf{q}}={\bf q}(\alpha,\beta,w)=\left[
		\begin{array}{c}
		-e^{-i\alpha}/2 \\
		we^{i(-\alpha/2+\beta)} \\
		1 \\
		\end{array}
		\right],
		\end{equation}
		where $\alpha\in [-\pi/2,\pi/2]$, $\beta\in [0, \pi)$ and $w\in [-\sqrt{\cos(\alpha)},\sqrt{\cos(\alpha)}]$.
    Because ${\bf{q}}\in I^{\star}_0$, then $$|\langle {\bf{q}}, (q_{\infty}) \rangle |^2 =|\langle {\bf{q}},  B^{-2} (q_{\infty}) \rangle |^2=1.$$ 
    We calculate  
    $|\langle {\bf{q}}, B^{-1} (q_{\infty}) \rangle |^2 $ to be 
		\begin{flalign} \label{equation:I0starI0plus}
	 \begin{aligned}
	w^2+\frac{1+\cos(\alpha)}{2}+w\left(\sin\left(\frac{\alpha-2 \beta}{2}\right)-\sin\left(\frac{\alpha+2 \beta}{2}\right)\right),
		\end{aligned}
		\end{flalign}
    and  
    $|\langle {\bf{q}}, B (q_{\infty}) \rangle |^2 $
    to be 
			\begin{flalign} \label{equation:I0starI0minus}
		\begin{aligned}
		w^2+\frac{1+\cos(\alpha)}{2}-w\left(\sin\left(\frac{\alpha-2 \beta}{2}\right)-\sin\left(\frac{\alpha+2 \beta}{2}\right)\right).
		\end{aligned}
		\end{flalign}
    Then the triple intersection $I_{0}^{+} \cap I^{-}_{0} \cap I^{\star}_{0}$ is given by 
	 \begin{flalign} \label{equation:I0starI0minusI0plus}
	\left\{ \begin{aligned}
	w^2+\frac{1+\cos(\alpha)}{2}=1; \qquad& \  \\ 
		w\left(\sin\left(\frac{\alpha-2 \beta}{2}\right)-\sin\left(\frac{\alpha+2 \beta}{2}\right)\right)=0. \qquad& \  
	\end{aligned}
	\right.
	\end{flalign}
    The solutions are $$\alpha=\pi, \quad \text{or}\quad w=0, \quad \text{or}\quad \beta=0.$$
	Since $\alpha\in [-\pi/2,\pi/2]$, we know that $\alpha=\pi$ is not a solution of  our equation. 
	If $w=0$, then $\alpha= 0$, which corresponds to the point $\mathbf{p}_{B}$. 
    If $\beta=0$, the solutions lie in a meridian of $I^{\star}_0$. 
    In this case, the set ${\bf{q}}(\alpha,0,w)$ is a union of two geodesics which are crossed at $\mathbf{p}_{B}$, 
    where $w = \pm \sqrt{\frac{1-\cos(\alpha)}{2}}$.  

   Moreover, in the case of $\beta=0$, we consider the intersection of $I_{0}^{+}\cap I^{-}_{0}\cap I^{\star}_{0}$ with $\partial{\bf H}^2_{\mathbb C}$.
    Note that $|\langle {\bf{q}}, {\bf{q}} \rangle |^2 =w^2-\cos(\alpha)$.
    If $\bf{q}\in\partial{\bf H}^2_{\mathbb C}$,
    we have $w=\pm \sqrt{\cos(\alpha)}$.
    Hence the equation $\pm \sqrt{\cos(\alpha)}= \pm \sqrt{\frac{1-\cos(\alpha)}{2}}$  
    gives the intersection point.
    Specifically, $$\alpha= \pm \arccos(\frac{1}{3}),~~\beta=0,~~~ w= \pm \frac{1}{3^{1/2}}.$$ 	
  The coordinates of the points $w_i$ for $i=1, 2, 3, 4$ are available.
\end{proof}

From Proposition \ref{prop:4pppGirauddisk}, we have the following three propositions on the combinatorics of 3-faces $s_{0}^{+}$,  $s_{0}^{-}$ and $s_{0}^{\star}$.
The readers are referred to Figure \ref{figure:3side4ppp} for them. 
\begin{prop}\label{prop:4ppp3face0plus}
For the 3-face  $s_{0}^{+}$ of $D_{\Sigma}$:
\begin{enumerate}[(i)]
    \item $s_{0}^{+} \cap s_{0}^{-}$ is topologically the union of two sectors.
    \item  $s_{0}^{+} \cap s^{\star}_{0}$ is topologically the union of two sectors.
    \item  $\partial s_{0}^{+} \cap \partial {\bf H}^2_{\mathbb C}$ is a disk.
    \item $s_{0}^{+}$ is a 3-ball in $\overline{\hc}$, and $\partial s_{0}^{+}$ is the union of $\partial s_{0}^{+} \cap \partial {\bf H}^2_{\mathbb C}$,  $s_{0}^{+} \cap s_{0}^{-}$ and  $s_{0}^{+} \cap s^{\star}_{0}$. 
\end{enumerate}	
\end{prop}

\begin{proof}The side $s_{0}^{+}$ is contained in  the isometric sphere $I_{0}^{+}$. Besides $I_{0}^{+}$ is tangent to $I_{-1}^{+}$ and  $I_{1}^{+}$ at a parabolic fixed point respectively, the isometric sphere $I_{0}^{+}$ doesn't intersect with all the other isometric spheres except $I_{0}^{-}$ and $I_{0}^{\star}$. By Proposition \ref{prop:4pppGirauddisk}, we have both  $I_{0}^{+} \cap I_{0}^{-}$   and 
 $I_{0}^{+} \cap I^{\star}_{0}$ are  topologically the union of two sectors.  We also have that $I_{0}^{+} \cap I_{0}^{-}$  and 
 $I_{0}^{+} \cap I^{\star}_{0}$ are glued together along  $I_{k}^{+}\cap I_{k}^{-} \cap I^{\star}_0$ to a disk in the boundary of  $s_{0}^{+}$. The boundary of this disk bounds a disk in $I^{+}_0 \cap {\bf H}^2_{\mathbb C}$. This  together co-bound a 3-ball in $I^{+}_0$, which is $s^{+}_0$. 
 \end{proof}
		
The proof of Proposition \ref{prop:4ppp3face0minus} is similar, we omit it.
\begin{prop}\label{prop:4ppp3face0minus}
For the 3-face $s_{0}^{+}$ of $D_{\Sigma}$:
\begin{enumerate}[(i)]
    \item $s_{0}^{-} \cap s_{0}^{+}$ is topologically the union of two sectors.	
    \item  $s_{0}^{-} \cap s^{\star}_{0}$ is topologically the union of two sectors.
    \item  $\partial s_{0}^{-} \cap \partial {\bf H}^2_{\mathbb C}$ is a disk.
    \item $s_{0}^{-}$ is a 3-ball in $\overline{\hc}$, and $\partial s_{0}^{-}$ is the union of $\partial s_{0}^{-} \cap \partial {\bf H}^2_{\mathbb C}$, $s_{0}^{-} \cap s_{0}^{+}$ and  $s_{0}^{-} \cap s^{\star}_{0}$.
\end{enumerate}
\end{prop}

\begin{prop}\label{prop:4ppp3face0star}
For the 3-face  $s^{\star}_{0}$ of $D_{\Sigma}$:
\begin{enumerate}[(i)]
	\item $s^{\star}_0 \cap s_{0}^{+}$ is topologically the union of two sectors.
		
	\item  $s^{\star}_0 \cap s^{-}_{0}$ is topologically the union of two sectors.
		
	\item  $\partial s_{\star}^{-} \cap \partial {\bf H}^2_{\mathbb C}$ is the union of two disjoint disks.
			
	\item $s^{\star}_{0}$ is a topologically solid light cone in $\overline{\hc}$, and $\partial s_{0}^{+}$ is the union of $\partial s_{0}^{+} \cap \partial {\bf H}^2_{\mathbb C}$, $s^{\star}_0 \cap s_{0}^{-}$ and $s^{\star}_{0} \cap s_{0}^{+}$. 
	\end{enumerate}
\end{prop}

\begin{proof}
The side $s_{0}^{\star}$ is contained in  the isometric sphere $I_{0}^{\star}$. Besides $I_{0}^{\star}$ is tangent to $I_{-1}^{\star}$ and  $I_{1}^{\star}$ at a parabolic fixed point respectively, the isometric sphere $I_{0}^{\star}$ doesn't intersect with all the other isometric spheres except $I_{0}^{+}$ and $I_{0}^{-}$. By Proposition \ref{prop:4pppGirauddisk}, we have both  $I_{0}^{\star} \cap I_{0}^{+}$   and 
$I_{0}^{\star} \cap I^{-}_{0}$ are  topologically the union of two sectors.  We also have that $I_{0}^{\star} \cap I_{0}^{+}$   and 
$I_{0}^{\star} \cap I^{-}_{0}$  are glued together along $I_{k}^{+}\cap I_{k}^{-} \cap I^{\star}_0$ to a light cone on the boundary of $s_{0}^{\star}$. 
The boundaries of this light cone are two simple curves. 
Each of them bounds a disk in $I^{\star}_0 \cap {\bf H}^2_{\mathbb C}$.  
The light cone and the two disks together co-bound a topologically solid light cone in $I^{\star}_0$, which is $s^{\star}_0$. 
\end{proof}

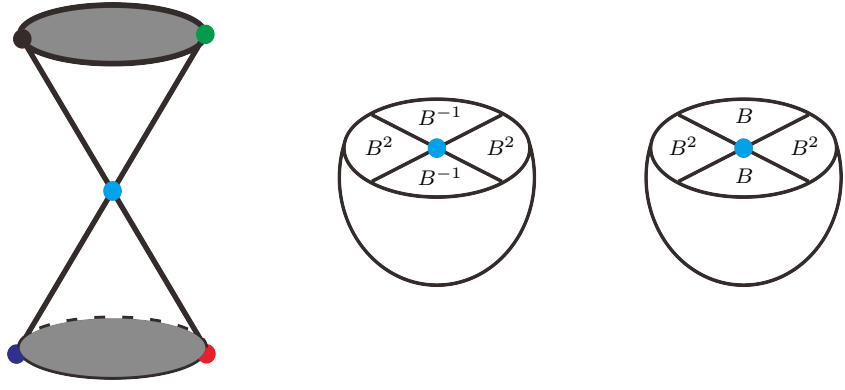
\begin{figure}[htbp]
	\begin{center}
		\begin{tikzpicture}
		\node at (0,0) {\includegraphics[width=11cm,height=5.0cm]{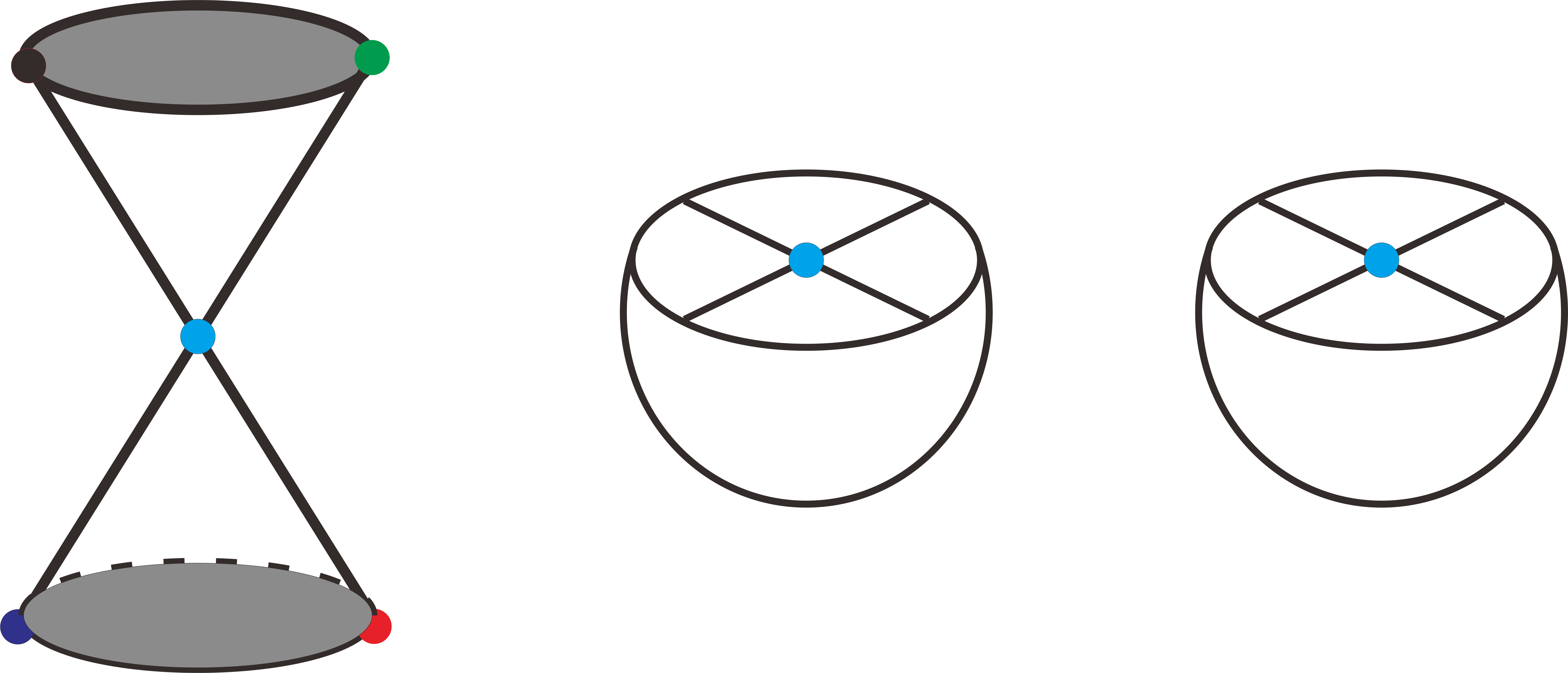}};
		
		\node at (0.2,0.2){\tiny $B^{-1}$};
		\node at (0.2,1.0){\tiny $B^{-1}$};
		
		\node at (-0.6,0.6){\tiny $B^{2}$};
		\node at (1,0.6){\tiny $B^{2}$};

	\node at (4.2,0.2){\tiny $B$};
	\node at (4.2,1.0){\tiny $B$};
	
	\node at (5.0,0.6){\tiny $B^{2}$};
	\node at (3.4,0.6){\tiny $B^{2}$};
		
		\end{tikzpicture}
	\end{center}
	\caption{3-sides of $D_{\Sigma}$. The left one is  $s^\star_0$, the two shaded disks are $s^\star_0 \cap \partial {\bf H}^2_{\mathbb C}$.
		 The central one is  $s^{+}_0$,  the $B^{-1}$-labeled  sectors are $s^{+}_0 \cap s^{-}_0$,  the $B^{2}$-labeled  sectors are $s^{+}_0 \cap s^{\star}_0$. The right one is $s^{-}_0$, the labelled sectors have similar meanings. The cyan point in each sub-figure is the fixed point of $B$. }
	\label{figure:3side4ppp}
\end{figure}

The side pairing maps on $D_{\Sigma}$  are defined by: 
$$A^kBA^{-k}:s_{k}^{+}\longrightarrow s_{k}^{-},\quad  A^kB^{-1}A^{-k}:s_{k}^{-}\longrightarrow s_{k}^{+}, \quad  A^kB^{2}A^{-k}:s_{k}^{\star}\longrightarrow s_{k}^{\star}.$$

By the $A$-symmetry, we may only consider the following properties. 
\begin{prop}\label{prop:sideparing4pp}For the side-pairings  $B$ and  $B^2$:
\begin{enumerate}[(i)]
	\item The side pairing map  $B$ is a homeomorphism from $s_{0}^{+}$ to $s_{0}^{-}$. Furthermore, $B$ sends the ridge $s_{0}^{+} \cap s_{0}^{-}$ to $s_{0}^{0} \cap s_{0}^{\star}$, and sends the ridge $s_{0}^{+} \cap s_{0}^{\star}$ to $s_{0}^{-} \cap s_{0}^{\star}$.
	
    \item The side pairing map $B^2$ is a self-homeomorphism of $s_{0}^{\star}$, and it exchanges the ridges $s_{0}^{+} \cap s_{0}^{\star}$ and $s_{0}^{-} \cap s_{0}^{\star}$.
	\end{enumerate}
\end{prop}
\begin{proof}  
The ridge $s_{0}^{+}\cap s_{0}^{-}$
	is defined by the triple equality
	$$\langle z, \mathbf{ q}_{\infty}\rangle=\langle z, B^{-1}(\mathbf{ q}_{\infty})\rangle=\langle z, B(\mathbf{ q}_{\infty})\rangle.$$ 
	The map $B$ sends   $$\{\mathbf{ q}_{\infty}, B^{-1}(\mathbf{ q}_{\infty}), B(\mathbf{ q}_{\infty}) \}$$ to  $$\{B(\mathbf{q}_{\infty}), \mathbf{ q}_{\infty}, B^2(\mathbf{ q}_{\infty})\}.$$ So $B$   sends the ridge $s_{0}^{+} \cap s_{0}^{-}$ to $s_{0}^{0} \cap s_{0}^{\star}$. 
    Another claim in Proposition \ref{prop:sideparing4pp} can be proved similarly. 
\end{proof}

{\bf Proof of Theorem \ref{thm:fundamental-domain4ppp}:} 
After Propositions 	\ref{prop:4ppp3face0plus},	\ref{prop:4ppp3face0minus},	\ref{prop:4ppp3face0star} and  \ref{prop:sideparing4pp}.

\paragraph{{\bf Local tessellation}}
We prove the tessellations around the sides and ridges of  $D_{\Sigma}$.

(1). Since $A^kB^{\pm}A^{-k}$ sends the exterior of $I_k^{\pm}$ to the interior of $I_k^{\mp}$, we see that $D_{\Sigma}$ and $A^kB^{\pm}A^{-k}(D_{\Sigma})$ have disjoint interiors and cover a neighbourhood of each point in the interior of the 3-side $s_k^{\mp}$.

(2). For the ridge  $s_{0}^{+} \cap s_{0}^{\star}$,  the ridge circle is
\begin{flalign}
\nonumber  s_{0}^{+} \cap s_{0}^{\star} \xrightarrow{B} s_{0}^{+} \cap s_{0}^{-} \xrightarrow{B} s_{0}^{\star} \cap s_{0}^{\star}  \xrightarrow{B^2}  s_{0}^{+} \cap s_{0}^{\star}.
\end{flalign}
This gives the relation $B^4=id$;
By a standard argument as in \cite{ParkerWill:2017}, we have that $D_{\Sigma} \cup B^2(D_{\Sigma}) \cup  B^{-1}(D_{\Sigma})$ will cover a small neighborhood of $s_{0}^{+} \cap s_{0}^{\star}$.

\paragraph{{\bf Completeness}}

We must construct a system of consistent horoballs at the parabolic fixed points.
First, we consider the side pairing maps on the parabolic fixed points $x_1=p_{AB}$,  $x_2=p_{AB^2}$ and $x_3=p_{AB^{-1}}$.
We have 
\begin{eqnarray*}
	B &:&   p_{AB} \longrightarrow p_{BA},\\
	A&:&   p_{BA} \longrightarrow p_{AB},\\
	B^2 &:&   p_{AB^{2}} \longrightarrow p_{B^2A} ,\\
	A  &:&   p_{B^{2}A} \longrightarrow  p_{AB^{2}},\\
B^{-1} &:&   p_{AB^{-1}}  \longrightarrow  p_{B^{-1}A},\\
A &:&   p_{B^{-1}A} \longrightarrow  p_{AB^{-1}}.
\end{eqnarray*}

Up to powers of $A$,  the cycles for the parabolic fixed points are the following
\begin{eqnarray*}
	&&p_{AB} \xrightarrow{B} p_{BA}\xrightarrow{A} p_{AB},\\
	&&p_{AB^2} \xrightarrow{B^2} p_{B^2A}\xrightarrow{A} p_{AB^2},\\&& p_{AB^{-1}} \xrightarrow{B^{-1}} p_{BA^{-1}}\xrightarrow{A} p_{AB^{-1}}.
\end{eqnarray*}
That is,  $p_{AB}$, $p_{AB^2}$ and $p_{AB^{-1}}$ are fixed by $AB$, $AB^2$ and $AB^{-1}$ respectively.
The elements $AB$, $AB^2$ and $AB^{-1}$ are unipotent and  preserve all horoballs at $p_{AB}$, $p_{AB^2}$ and $p_{AB^{-1}}$ respectively.

This ends the proof of Theorem \ref{thm:fundamental-domain4ppp}.

\section{Topology of the ideal boundary of the Ford domain 
\texorpdfstring{$D_{\Sigma}$}{} %
}
\label{section:topology4pppford}

In this section, we study the topology at infinity of the Ford domain $D_{\Sigma}$ of $ \Sigma < \Delta_{4,\infty,\infty;\infty}$.
Take a singular surface 
$$\mathcal{S} = \bigcup_{k \in \mathbb{Z}} ((s_k^{+} \cup s_k^{-} \cup s_k^{ \star})\cap \partial \hc)$$  
in the Heisenberg group.
Part of $\mathcal{S}$ is shown in Figures \ref{figure:4pppford} and \ref{figure:4pppfordabstract}.
The ideal boundary of a Ford domain is the intersection of the exteriors of all spinal spheres. 
So, the ideal boundary of $D_{\Sigma}$,
$$\partial_{\infty} D_{\Sigma}=D_{\Sigma} \cap \partial \hc,$$ 
is the region outside $\mathcal{S}$ in Figures  \ref{figure:4pppford} and \ref{figure:4pppfordabstract}.
We will show that $\partial_{\infty} D_{\Sigma}$ is an $A$-invariant infinite genus
handlebody.

Figure \ref{figure:4pppford} indicates that there are "holes" in $\partial_{\infty} D_{\Sigma}$. 
One "hole" with purple boundary in Figure \ref{figure:4pppfordabstract} is enclosed by the spinal spheres of $I^{+}_0$, $I^{\star}_0$, $I_{1}^{-}$ and $I_{1}^{\star}$, labeled as $B$, $B^2$, $AB^{-1}A^{-1}$ and $AB^2A^{-1}$ respectively.    
The "holes" make the topology of
$\partial_{\infty}D_{\Sigma}$ dramatically different from those studied in \cite{Deraux:2016, jwx, MaXie2020, ParkerWill:2017}.  
The ideal boundaries of the Ford domains in this literature are all infinite cylinders. 
Since the topology of $\partial_{\infty} D_{\Sigma}$ is complicated, we need explicit disks to cut out a 3-ball from it in our proof of Theorem \ref{thm:4pp}.

\subsection{Two disks 
\texorpdfstring{$E_B$ and $E_{B^{-1}}$}{} %
} 

We will construct a set of disks 
$$\{A^{k}(E_B \cup E_{B^{-1}})\}_{k \in \mathbb{Z}},$$ 
which capping off all the  "holes"  in  $\partial_{\infty} D_{\Sigma}$. 
By cutting $\partial_{\infty} D_{\Sigma}$ along these disks we get an infinite cylinder, 
which is a solid torus topologically. 

Recall the points $x_{i}$ for $i=1,2,3,4$ in Section \ref{section:ford}.  We also take two points
\begin{equation*}
y_1=\left[\begin{matrix}
-\frac{1}{2} \\ -\frac{\sqrt{3}-i}{2}  \\  1 \end{matrix}\right] \in \partial_{\infty}I^{+}_0\cap \partial_{\infty}I^{\star}_0, \quad
\quad
y_2=\left[\begin{matrix}
-\frac{1}{2} \\ -\frac{\sqrt{3}+i}{2} \\ 1\end{matrix}\right] \in \partial_{\infty}I^{-}_0\cap \partial_{\infty}I^{\star}_0.
\end{equation*}
Now consider a topological circle $C_{B}$ consisting of four arcs:
\begin{itemize}
\item Let $C_5$ be the arc in $s^{+}_0 \cap \partial \mathbf{H}^2_{\mathbb C}$ joining 
$x_1$ to $y_1$ given by
\begin{equation*}\label{c5}
    C_5=\left\{\begin{bmatrix}
-\frac{1}{2}-e^{2si}-i\sqrt{2\cos{(2s)}}e^{\frac{2s-\pi}{4}i} \\ 
-\sqrt{2\cos{(2s)}}e^{\frac{2s-\pi}{4}i}+i \\ 
1\end{bmatrix}
:  s\in [\frac{5\pi}{6},\pi]\right\}.
\end{equation*}
\item Let $C_1$ be the $\mathbb{C}$-arc in $s^{\star}_0 \cap \partial \mathbf{H}^2_{\mathbb C}$ joining 
 $y_1$ to $x_2$ given by 
 \begin{equation*}\label{c1}
    C_1=\left\{\begin{bmatrix}
-\frac{1}{2} \\ 
e^{is} \\ 
1\end{bmatrix}
:  \ s\in [\frac{5\pi}{6},\pi]\right\}.
\end{equation*}
\item Let $C_6$ be the arc in $s^{-}_1 \cap \partial \mathbf{H}^2_{\mathbb C}$ joining 
 $AB^2(y_1)$ and $x_1$ given by 
 \begin{equation*}
    C_6=\left\{\begin{bmatrix}
-\frac{5}{2}-2i-e^{2si}+(2-i)\sqrt{2\cos{(2s)}}e^{\frac{2s-\pi}{4}i} \\ 
-2-i+\sqrt{2\cos{(2s)}}e^{\frac{2s-\pi}{4}i} \\ 
1\end{bmatrix}
: \ s\in [\frac{5\pi}{6},\pi]\right\}.
\end{equation*}
\item  Let $C_2$ be the $\mathbb{C}$-arc in $s^{\star}_1 \cap \partial \mathbf{H}^2_{\mathbb C}$ joining 
 $x_2$ and $AB^2(y_1)$ given by 
 \begin{equation*}
    C_2=\left\{\begin{bmatrix}
-\frac{5}{2}-2e^{is} \\ 
-2-e^{is} \\ 
1\end{bmatrix}
: \ s\in [\frac{5\pi}{6},\pi]\right\}.
\end{equation*}
\end{itemize}

We will show later that the union of $C_5$, $C_1$, $C_6$ and $C_2$ is a simple closed curve, say $C_B$. We use ruled surfaces to construct a disk $E_B$ with boundary $C_{B}$.
One of the ruled surfaces with directrices $C_5$ and $C_1$, is denoted by $E_B^l$. Another one with directrices $C_6$ and $C_2$, is denoted by $E_B^r$.
Let $c_i$ represents the coordinates of $C_i$ in $\mathcal{N}$ for $i=1,2,5,6$. 
The ruled surfaces can be described as
\begin{equation*}
        E_B^l=(1-a)\cdot c_5+a\cdot c_1,\quad
        E_B^r=(1-a)\cdot c_6+a\cdot c_2
\end{equation*}
for $a\in[0,1]$. 
In Figure \ref{figure:E1E2}, $E_B^l$ appears as the blue part, and $E_B^r$ appears as the green part.
We define $E_B=E_B^l\cup E_B^r$.
It is easy to verify that $E_B^l\cap E_B^r$ is the straight-line segment from $x_1$ to $x_2$ in $\mathcal{N}$.
In Figure \ref{figure:capoffwithEB}, $E_{B}$ patches up the "hole" of $\partial_{\infty} D_{\Sigma}$ enclosed by the spinal spheres  $\partial_{\infty}I_0^{+}$, $\partial_{\infty}I_1^{-}$, $\partial_{\infty}I^{\star}_0$ and $\partial_{\infty}I^{\star}_1$.

\begin{figure}[tbp]
\centering
\begin{minipage}{.53\textwidth}
  \centering
  \begin{overpic}[width=7cm,height=6cm]{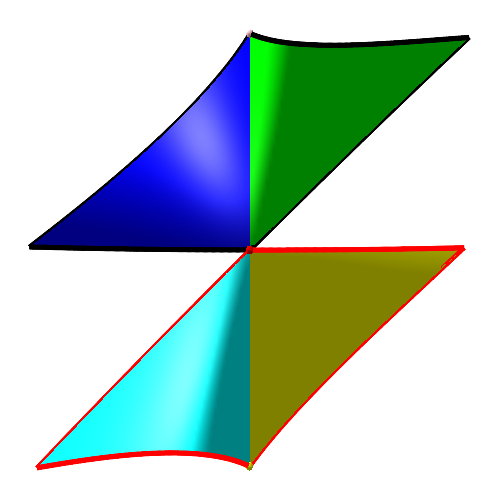}
  \put(82,81.05){\small{$AB^2(y_1)$}}
  \put(82,44){\small{$AB^2(y_2)$}}
  \put(0,42.05){\small{$y_1$}}
  \put(2,2.05){\small{$y_2$}}
  \put(42,40){\small{$x_2$}}
  \put(43,80.05){\small{$x_1$}}
  \put(44,2){\small{$x_3$}}
  \end{overpic}
  \caption{The disks $E_{B}$ (the union of the blue and green triangles)
    and $E_{B^{-1}}$ (the union of the cyan and yellow triangles).}
  \label{figure:E1E2}
\end{minipage}%
\hfill
\begin{minipage}{.47\textwidth}
  \centering
  {\includegraphics[width=6.5cm,height=6cm]
  {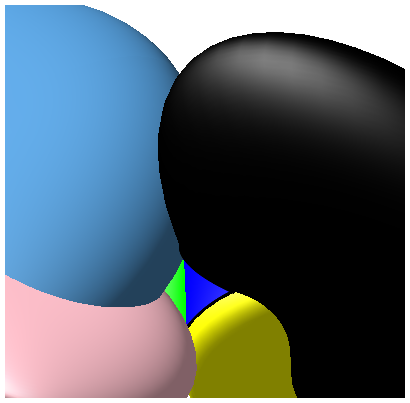}}
  \caption{The disk $E_{B}$ caps off the "hole" of $\partial_{\infty} D_{\Sigma}$ enclosed by the spinal spheres $\partial_{\infty}I_0^{+}$ (black), $\partial_{\infty}I_1^{-}$ (steelblue), $\partial_{\infty}I^{\star}_0$ (yellow) and $\partial_{\infty}I^{\star}_1$ (pink).}
  \label{figure:capoffwithEB}
\end{minipage}
\end{figure}

As an example, we give the  Heisenberg coordinates of $E_B^l$:
\begin{equation}\label{xyt}
    \begin{cases}
        x=\cos{(s)}a -\sqrt{2\cos{(2s)}}\sin{\frac{2s+\pi}{4}}\cdot(1-a)\\
        y=\sin{(s)}a+\Big(1+\sqrt{2\cos{(2s)}}\cos{\frac{2s+\pi}{4}}\Big)\cdot(1-a)\\
        t=-2\Big(\sin{(2s)}+ \sqrt{2\cos{(2s)}}\sin{\frac{2s+\pi}{4}}\Big)\cdot(1-a),
\end{cases}
\end{equation}
where $a\in[0,1]$ and $s\in[\frac{5\pi}{6},\pi]$.
We use projection mapping to give results on whether $E_B^l$ intersects the spinal spheres.
\begin{lem}
    Let $\Pi:\mathcal{N}\longrightarrow\mathbb{C}$ be the nature projection.
    The union of $\Pi(C_5)$ and $\Pi(C_1)$ is a simple closed curve in $\mathbb{C}$ contained in $\Pi(\partial_{\infty}I_0^{-})$ and $\Pi(\partial_{\infty} I_1^{-})$. 
    Besides, 
    \begin{enumerate}[(i)]
    \item \label{negative01}
    $\Pi(E_B^l)$ intersects $\Pi(\partial_{\infty}I_0^{-})$ and $\Pi(\partial_{\infty} I_1^{-})$ at the same point $\Pi(x_1)$.
    \item \label{star0c1}
    $\Pi(E_B^l)$ intersects $\Pi(\partial_{\infty}I_0^{*})$ to be the arc $\Pi(C_1)$.
    \item \label{kkk}
    $\Pi(E_B^l)$ intersects $\Pi(\partial_{\infty}I_k^{*}\cup \partial_{\infty}I_k^{+}\cup \partial_{\infty}I_k^{-})_{k\in \mathbb{Z}}$ to be the empty set for $k\neq 0,1$.
    \end{enumerate}
\end{lem}
\begin{proof}  
    According to the equations of $C_5$ and $C_1$,  it is easy to know $$\Pi(C_5)\cap \Pi(C_1)=\Pi(y_1) \cup \Pi(x_1).$$
    So $\Pi(C_5) \cup \Pi(C_1)$ is a simple closed curve. 
    It is also easy to check that $\Pi(C_5)$ and $\Pi(C_1)$ are all contained in $\Pi(\partial_{\infty}I_1^{+})$ and $\Pi(\partial_{\infty}I_0^{+})$.
    We have completed the first half of the proof. 
    
    For \ref{negative01}, according to $$\Pi(x_1)=\Pi(x_2)=\Pi(\partial_{\infty}I_0^{+})\cap\Pi(\partial_{\infty}I_1^{-})=\Pi(\partial_{\infty}I_0^{-})\cap\Pi(\partial_{\infty}I_1^{+}),$$ 
    and the affine convexity of $\Pi(\partial_{\infty}I_0^{+})$ and $\Pi(\partial_{\infty}I_1^{+})$, we know that 
    $$\Pi(E_B^l)\cap\Pi(\partial_{\infty}I_0^{-})=\Pi(E_B^l)\cap\Pi( \partial_{\infty}I_1^{-})=\Pi(x_1).$$
  
    For \ref{star0c1}, we prove that for a given $s\in[\frac{5\pi}{6},\pi]$, the projection of each straight line on $E_B^l$, with parameterization (\ref{xyt}), intersects $\Pi(C_1)$ 
    at one point. 
    Write $d_u=\big(\cos{(s)},\sin{(s)}\big)$, which is the coordinate of some point in $\Pi(C_1)$. 
    It is also the vector from the origin $O$ in $\mathbb{C}$ to this point.
    Write $d_l=\big(\frac{\partial x}{\partial a},\frac{\partial y}{\partial a}\big)\big |_{a=1}$, which is the direction of the projection of the straight line on $E_B^l$.
    Note that when $a=0$, the corresponding point of $E_B^l$ is in $C_5$, and when $a=1$, the corresponding point is in $C_1$.
    By calculating the inner product of $-d_l$ and $d_u$, we can determine whether the point in the straight line of $\Pi(E_B^l)$ is far away from $\Pi(C_1)$ when $a$ decreases from 1 to 0.
    In fact,  the inner product of $-d_l$ and $d_u$ is
       \begin{equation*}
        \begin{aligned}
            &\big(\sin{(\frac{s}{2})}-\cos{(\frac{s}{2})}\big)\sqrt{\cos{(2s)}}+ \sin{(s)}-1\\
            &=-\sqrt{2}\sin{(\frac{\pi}{4}-\frac{s}{2})}  \cdot \big(\sqrt{2}\sin{(\frac{\pi}{4}-\frac{s}{2})}+\sqrt{\cos{(2s)}}\big),
        \end{aligned}
    \end{equation*}
  which is non-negative when $s\in[\frac{5\pi}{6},\pi]$.
   By the affine convexity of
$\Pi( \partial_{\infty}I^{\star}_0)$, we have $\Pi(E_B^l)\cap \Pi(\partial_{\infty}I^{\star}_0)=\Pi(C_1)$.

   For \ref{kkk}, because
$$\Pi(E_B^l)\subset\Pi(\partial_{\infty}I_0^{-})\cap\Pi( \partial_{\infty}I_1^{-}),$$
   which is outside all $\Pi(\partial_{\infty}I_k^{*}\cup \partial_{\infty}I_k^{+}\cup \partial_{\infty}I_k^{-})_{k\in \mathbb{Z}}$ for $k\neq 0,1$, we get the conclusion.
\end{proof}

Instead of drawing entire projections of $E_B^l$ and $E_B^r$, we draw the projections of several straight lines on them, represented by blue and green line segments respectively in Figure \ref{figure:proj}.
The projections of $E_{B}$ and $E_{B^{-1}}$ are the union of blue and green crescents and the union of yellow and cyan crescents respectively.
The projection of $\partial_{\infty} I_0^{+}$, $\partial_{\infty} I_0^{-}$, $\partial_{\infty} I_1^{+}$ and $\partial_{\infty} I_1^{-}$ are the round disk with radius $\sqrt{2}$ and centered at $i$, $-i$, $-2+i$ and $-2-i$, respectively.
The projections of $\partial_{\infty} I_{0}^{\star}$ and $\partial_{\infty} I_1^{\star}$ are round disks with radius $1$ and centered at $0$ and $-2$ respectively. 
The projection of $F$ is the union of magenta, coral and navy curves.
Figure \ref{figure:proj51} gives an enlarged view of $\Pi(E_B^l)$.

\begin{figure}[tbp]
\centering
\begin{minipage}{.52\textwidth}
  \centering
  \begin{overpic}[width=6cm,height=7.5cm]{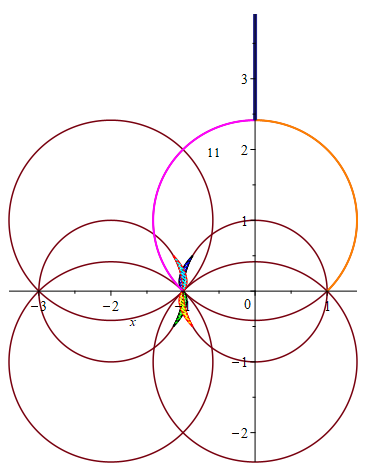}
  \put(56.5,85){\small{$\Pi(F_3)$}}
  \put(76.5,60){\small{$\Pi(F_1)$}}
  \put(22,60){\small{$\Pi(F_2)$}}
  \end{overpic}
  \caption{The projections of $\partial_{\infty} I_0^{\pm}$, $\partial_{\infty}I_1^{\pm}$, $\partial_{\infty}I_{0}^{\star}$, $\partial_{\infty}I_1^{\star}$, $E_{B}$, $E_{B^{-1}}$ and $F$.}
  \label{figure:proj}
\end{minipage}%
\hfill
\begin{minipage}{.48\textwidth}
  \centering
  \begin{overpic}[width=5.5cm,height=7.3cm]{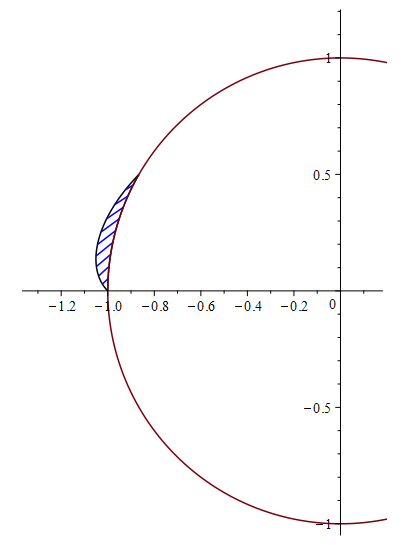}
  \put(20.8,51){\small{$\Pi(C_1)$}}
  \put(2,51){\small{$\Pi(C_5)$}}
  \end{overpic}
  \caption{An enlarged view of $\Pi(E_B^l)$ and part of $\Pi(\partial_{\infty}I^{\star}_0)$.}
  \label{figure:proj51}
\end{minipage}
\end{figure}

For the relationship between $E_B^l$ and the remaining spinal spheres, we have:
\begin{lem}
    For $E_B^l\cap(\partial_{\infty}I_0^{+}\cup \partial_{\infty}I_1^{+}\cup \partial_{\infty}I_0^{*})$, we have 
    \begin{enumerate}[(i)]
        \item \label{caps0}$E_B^l\cap \partial_{\infty}I_0^{+}=C_5$.
        \item \label{caps1*}$E_B^l\cap \partial_{\infty}I_1^{*}=x_2$.
        \item \label{caps1+}$E_B^l\cap \partial_{\infty}I_1^{+}=\emptyset$.
    \end{enumerate}
\end{lem} 

\begin{proof}
Let $p(s,a)=\big(x(s,a),y(s,a),t(s,a)\big)\in E_B^l$ be the point given by Equation (\ref{xyt}).
For \ref{caps0}, consider the Cygan distance  between $B^{-1}(q_{\infty})$ and $p(s,a)$. 
Let 
$$g(s,a)=d_{\textrm{Cyg}}^4\big(B^{-1}(q_{\infty}), p\big)=\big(x^2+(y-1)^2\big)^2+(t-2x)^2$$ 
with $a\in[0,1]$ and $s\in[\frac{5\pi}{6},\pi]$. 
Then
\begin{equation*}
    \begin{aligned}
        \frac{\partial g(s,a)}{\partial a}\Big |_{a=0}&=16\cos^{3/2}{(2s)}\big(\sin{\frac{s}{2}}-\cos{\frac{s}{2}}\big) +4\big(\sin{(3s)}-\cos{(4s)}+\sin{(s)}\big)-12\\
        &\ge 0,
    \end{aligned}
\end{equation*}
for $s\in[\frac{5\pi}{6},\pi]$ by Lemma \ref{inequalityabouts} later. 
By the affine convexity of the isometric spheres, we know that for $a\in[0,1]$ and $s\in[\frac{5\pi}{6},\pi]$,
$$d_{\textrm{Cyg}}\big(B^{-1}(q_{\infty}), p(s,a) \big)\ge d_{\textrm{Cyg}}\big(B^{-1}(q_{\infty}), p(s,0) \big),$$
and the equality holds if and only if $a=0$, which are points on $C_5$. 

For \ref{caps1*}, let $$h(s,a)=d_{\textrm{Cyg}}^4\big(AB^2A^{-1}(q_{\infty}), p(s,a) \big)=((x+2)^2+y^2)^2+(t-2y)^2,$$ where $a\in[0,1]$ and $s\in[\frac{5\pi}{6},\pi]$. 
By Lemma \ref{inequalityabouts} later, we have
\begin{equation*}
    \begin{aligned}
        \frac{\partial h(s,a)}{\partial a}\Big |_{a=1}=&\big(96\cos^3{\frac{s}{2}}+12\cos{\frac{s}{2}}+ 4\sin{\frac{s}{2}}\big)\sqrt{\cos{(2s)}}+ 52\cos{(s)}-28\sin{(s)} \\
        &\qquad+ 12\cos{(2s)}+4\cos{(3s)}- 8\sin{(2s)}+ 40
        \le 0,
    \end{aligned}
\end{equation*}  
for $s\in[\frac{5\pi}{6},\pi]$. 
By the affine convexity of the isometric spheres, we know that for $a\in[0,1]$ and $s\in[\frac{5\pi}{6},\pi]$,
$$d_{\textrm{Cyg}}\big(AB^2A^{-1}(q_{\infty}), p(s,a) \big)\ge d_{\textrm{Cyg}}\big(AB^2A^{-1}(q_{\infty}), p(s,1) \big)\ge1.$$
The first equality holds if and only if $a=1$, and the second equality holds if and only if $s=\pi$, corresponding to the point $x_2$. 

For \ref{caps1+}, consider the plane
$$ -3 (X + 1 + Y)+T=0$$ 
in the Heisenberg group. 
On the one hand, it is easy to check that $t-3(x + 1 + y)\le0$, and the equality holds if and only if $p(s,a)=x_2$.
On the other hand, 
    $$\partial_{\infty}I_1^{+}=\left\{\begin{bmatrix}
2i-e^{-\alpha i}\pm(2+i)\sqrt{2\cos{\alpha}}\cdot e^{(-\frac{\alpha}{2}+\beta)i}-\frac{5}{2} \\ 
-2+i\pm\sqrt{2\cos{\alpha}}\cdot e^{(-\frac{\alpha}{2}+\beta)i} \\ 
1\end{bmatrix}
: \ \alpha\in [-\frac{\pi}{2},\frac{\pi}{2}],\ \beta\in [0,\pi)\right\}.
$$
Assume $p_0\in \partial_{\infty}I_1^{+}$ with Heisenberg coordinates $(x_0,y_0,t_0)$.
Then 
\begin{equation*}
    \begin{aligned}
        -3(x_0 + 1 + y_0)+t_0&=\mp\sqrt{2\cos{\alpha}}\big(\cos{(\frac{\alpha}{2}-\beta)}+\sin{(\frac{\alpha}{2}-\beta)}\big)+2\sin{\alpha}+4\\
        &\geq -2\sqrt{\cos{\alpha}}+2\sin{\alpha}+4>0.
    \end{aligned}
\end{equation*}
Hence, $E_B^l$ and $\partial_{\infty}I_1^{+}$ are separated by the above plane.
\end{proof}

\begin{lem}\label{inequalityabouts}
Let 
\begin{equation*}
    \begin{aligned}
        f_1(s)&=16\cos^{3/2}{(2s)}\big(\sin{\frac{s}{2}}-\cos{\frac{s}{2}}\big) +4\big(\sin{(3s)}-\cos{(4s)}+\sin{(s)}\big)-12\\
    \text{and\quad}    
    f_2(s)&=\big(96\cos^3{\frac{s}{2}}+12\cos{\frac{s}{2}}+ 4\sin{\frac{s}{2}}\big)\sqrt{\cos{(2s)}}+ 52\cos{(s)}-28\sin{(s)} \\
        &\qquad\qquad\qquad\qquad\qquad\qquad+ 12\cos{(2s)}+4\cos{(3s)}- 8\sin{(2s)}+ 40
    \end{aligned}
\end{equation*}
be two functions.
For $s\in[\frac{5\pi}{6},\pi]$, we have
   \begin{enumerate}[(i)]
    \item \label{inequalityabouts:1}
    $f_1(s)\ge0$, and the equality holds if and only if $s=\frac{5\pi}{6}$ or $s=\pi$.
    \item \label{inequalityabouts:2}
    $f_2(s)\le0$, and the equality holds if and only if $s=\frac{5\pi}{6}$ or $s=\pi$.
   \end{enumerate}
\end{lem}
\begin{proof} 
For \ref{inequalityabouts:1}, it is easy to see that $4\cos^{3/2}{(2s)}(\sin{\frac{s}{2}}-\cos{\frac{s}{2}})>0$ and $$-\sin{(3s)}+\cos{(4s)}-\sin{(s)}+3>0$$ when $s\in[\frac{5\pi}{6},\pi]$. 
Then 
    \begin{equation*}
    \begin{aligned}
        &\frac{f_1(s)}{2}\cdot\Big(4\cos^{3/2}{(2s)}(\sin{\frac{s}{2}}-\cos{\frac{s}{2}})+(-\sin{(3s)}+\cos{(4s)}-\sin{(s)}+3)\Big)\\
        &=-21 + 9\cos{(6s)} - 10\cos{(4s)} + 6\sin{(5s)} - 2\sin{(3s)} + 22\sin{(s)} \\
    &\qquad\qquad\qquad\qquad+ 23\cos{(2s)} - \cos{(8s)} - 2\sin{(7s)}.
    \end{aligned}
    \end{equation*}

Let $x=\sin{(s)}$, the above equation is written as:
$$F_1(x)=-32x(2x-1)(2x^6-x^5+2x^3-2x^2-x+1),$$
where $x\in[0,\frac{1}{2}]$.
Observe that $2x^6-x^5+2x^3=x^3(2x^3+2-x^2)>0$ and $1-2x^2-x> 0$ for $x\in (0,1/2)$. 
The roots of $F_1(x)$ in $[0,\frac{1}{2}]$ are only $0$ and $\frac{1}{2}$. 
Thus we have that $f_1(s)\ge0$, and the equality holds exactly when $s=\frac{5\pi}{6}$ or $s=\pi$. 

For \ref{inequalityabouts:2}, it is easy to see that $$(24\cos^3{\frac{s}{2}}+3\cos{\frac{s}{2}}+ \sin{\frac{s}{2}})\sqrt{\cos{(2s)}}>0$$ and $$13\cos{(s)}-7\sin{(s)}+4\cos{(2s)}+\cos{(3s)}- 2\sin{(2s)}+10<0$$ when $s\in[\frac{5\pi}{6},\pi]$. 
Then
\begin{equation*}
    \begin{aligned}
        &\frac{f_2(s)}{2}\cdot\Big((24\cos^3{\frac{s}{2}}+3\cos{\frac{s}{2}}+ \sin{\frac{s}{2}})\sqrt{\cos{(2s)}}\\
        &\qquad\qquad\qquad-\big(13\cos{(s)}-7\sin{(s)}+4\cos{(2s)}+\cos{(3s)}- 2\sin{(2s)}+10\big)\Big)\\
        &=-306 + 32\sin{(4s)} - \cos{(6s)} + 248\sin{(2s)} +95\cos{(4s)} \\ 
        &\qquad\qquad
        +4\sin{(5s)}+109\sin{(3s)}-268\cos{(s)} + 271\sin{(s)}\\
        &\qquad\qquad\qquad 
        +212\cos{(2s)} + 256\cos{(3s)} + 12\cos{(5s)}.
    \end{aligned}
    \end{equation*}

Let $x=\sin{(s)}$, the above equation is written as:
$$F_2(x)=32(x-\frac{1}{2})x\left((39-6x^2+5x)\sqrt{1-x^2} + x^4 + \frac{5x^3}{2}+\frac{47x^2}{2}-\frac{35x}{8}-\frac{309}{8}\right),$$
where $x\in[0,\frac{1}{2}]$.
Define
$$G(x)=(39-6x^2+5x)\sqrt{1-x^2} + x^4 + \frac{5x^3}{2}+\frac{47x^2}{2}-\frac{35x}{8}-\frac{309}{8}.$$ 
We claim that $G(x)>0$ for $x\in [0,\frac{1}{2}]$.
Note that
$\sqrt{1-x^2}\geq 1-\frac{x^2}{2}-\frac{x^4}{6}$ and $39-6x^2+5x>0$ for $x\in [0,\frac{1}{2}]$.
Then
\begin{equation*}
    \begin{aligned}
        G(x)&\geq (39-6x^2+5x)(1-\frac{x^2}{2}-\frac{x^4}{6})+x^4 + \frac{5x^3}{2}+\frac{47x^2}{2}-\frac{35x}{8}-\frac{309}{8}\\
        &=x^6- \frac{5}{6}x^5 -\frac{5}{2}x^4 - 2x^2 + \frac{5}{8}x + \frac{3}{8}\\
        &=(x-\frac{1}{2})^6+\frac{13}{6}x(x-\frac{1}{2})^4 -\frac{23}{12}x^4 -\frac{3}{4}x^3-\frac{89}{48}x^2 + \frac{65}{96}x + \frac{23}{64}\\
    \end{aligned}
\end{equation*}
It is easy to verify that 
$\frac{45}{96}x -\frac{23}{12}x^4 -\frac{3}{4}x^3\geq 0$ and $-\frac{89}{48}x^2 +\frac{20}{96}x+\frac{23}{64}\geq 0$ for $x\in [0,\frac{1}{2}]$.
Then we have $-\frac{23}{12}x^4 -\frac{3}{4}x^3-\frac{89}{48}x^2 + \frac{65}{96}x + \frac{23}{64}\geq 0$ for $x\in [0,\frac{1}{2}]$.
Now we have that $F_2(x)\le0$ and the roots of $F_2(x)$ in $[0,\frac{1}{2}]$ are only $0$ and $\frac{1}{2}$.
Hence $f_2(s)\le0$, and the equality holds exactly when $s=\frac{5\pi}{6}$ or $s=\pi$.
\end{proof}

Now, we have completed the proofs of \ref{EBdisk} and \ref{EBdisjoint} of Proposition \ref{prop:EB} below.
\begin{prop}\label{prop:EB} We have:
    \begin{enumerate}[(i)]
    \item \label{EBboundary} The union of $C_5$, $C_1$, $C_6$ and $C_2$ is a simple closed curve $C_B$.
    \item \label{EBdisk} $C_B$ bounds the disk $E_B$. 
    The interior of $E_B$ is disjoint from $\partial_{\infty}I_k^{*}\cup \partial_{\infty}I_k^{+}\cup \partial_{\infty}I_k^{-}$ for any $k \in \mathbb{Z}$.
    \item \label{EBdisjoint} $E_B \cap A^{k}(E_B)= \emptyset $ if $k \neq 0$. 
\end{enumerate}	
\end{prop}

\begin{proof} 
For \ref{EBboundary}, by Proposition \ref{prop:pair-disjoint4ppp},
we know that
$I^{\star}_{0}\cap I^{-}_{1}=\emptyset$,
$I^{\star}_{0}\cap I^{\star}_{1}=x_2$,
$I_{0}^{+}\cap I_{1}^{-}=x_1$,
and $I_{0}^{+}\cap I^{\star}_{1}=\emptyset$.
The curves $C_5$, $C_1$, $C_6$, and $C_2$ are contained in $I_{0}^{+}$, $I^{\star}_{0}$, $I^{\star}_{1}$ and $I^{-}_{1}$, respectively.
Therefore, we need only to check that $C_5\cap C_1=y_1$.
By the equations of $C_5$ and $C_1$, we have  $C_5 \cap C_1$ is given by 
\begin{equation*}
    \begin{bmatrix}
-\frac{1}{2}-e^{2si}-i\sqrt{2\cos{(2s)}}e^{\frac{2s-\pi}{4}i} \\ 
-\sqrt{2\cos{(2s)}}e^{\frac{2s-\pi}{4}i}+i \\ 
1\end{bmatrix}
=\begin{bmatrix}
-\frac{1}{2} \\ 
e^{s'} \\ 
1\end{bmatrix}
\end{equation*}
for some $s, s'\in [\frac{5\pi}{6},\pi]$.
This holds if and only if $s=s'=\frac{5\pi}{6}$, which corresponds to the point $y_1$. 
\end{proof}

Another topological circle $C_{B^{-1}}$ consisting of four arcs:
\begin{itemize}
    \item  Let $C_7$ be the arc in $s^{-}_0 \cap \partial \mathbf{H}^2_{\mathbb C}$ joining $y_2$ and $x_3$ given by 
    \begin{equation*}
        C_7=\left\{\begin{bmatrix}
        -\frac{1}{2}-e^{2si}+\sqrt{2\cos{(2s)}}e^{\frac{\pi+2s}{4}i} \\ 
        -i+i\sqrt{2\cos{(2s)}}e^{\frac{\pi+2s}{4}i} \\ 1\end{bmatrix}
        : \ s\in [0,\frac{\pi}{6}]\right\}.
    \end{equation*}
    \item Let $C_3$ be the $\mathbb{C}$-arc in $s^{\star}_0 \cap \partial \mathbf{H}^2_{\mathbb C}$ joining $x_2$ and $y_2$ given by
    \begin{equation*}
        C_3=\left\{\begin{bmatrix}
        -\frac{1}{2} \\ 
        -e^{si} \\ 
        1\end{bmatrix}
        : \ s\in [0,\frac{\pi}{6}]\right\}.
    \end{equation*}
    \item  Let $C_8$ be the arc in $s^{+}_1 \cap \partial \mathbf{H}^2_{\mathbb C}$ joining $x_3$ and $AB^2(y_2)$ given by 
    \begin{equation*}
        C_8=\left\{\begin{bmatrix}
        -\frac{5}{2}+2i-e^{2si}+(1-2i)\sqrt{2\cos{(2s)}}e^{\frac{\pi+2s}{4}i} \\ 
        -2+i-i\sqrt{2\cos{(2s)}}e^{\frac{\pi+2s}{4}i} \\ 
        1\end{bmatrix}
        : \ s\in [0,\frac{\pi}{6}]\right\}.
    \end{equation*}
    \item Let $C_4$ be the $\mathbb{C}$-arc in $s^{\star}_1 \cap \partial \mathbf{H}^2_{\mathbb C}$ joining $AB^2(y_2)$ and $x_2$ given by
    \begin{equation*}         
    C_4=\left\{\begin{bmatrix}         
    -\frac{5}{2} + 2e^{si} \\ 
    -2+e^{si} \\ 
    1 \end{bmatrix}
    : \ s\in [0,\frac{\pi}{6}]\right\}.
    \end{equation*}
\end{itemize}
Similarly, let $c_i$ represents the coordinates of $C_i$ in $\mathcal{N}$ for $i=3,4,7,8$.
Then construct two more ruled surfaces
\begin{equation*}
        E_{B^{-1}}^l=(1-a)\cdot c_8+a\cdot c_4,\quad
        E_{B^{-1}}^r=(1-a)\cdot c_7+a\cdot c_3
      \end{equation*} with 
$a\in[0,1]$.

Define $E_{B^{-1}}=E_{B^{-1}}^l\cup E_{B^{-1}}^r$.
We also have $E_{B^{-1}}^l\cap E_{B^{-1}}^r$ is the straight-line segment from $x_3$ to $x_2$ in $\mathcal{N}$.
In Figure \ref{figure:E1E2}, $E_{B^{-1}}^l$ is the cyan part, and $E_{B^{-1}}^r$ is the yellow part.
The proof of Proposition \ref{prop:EinverseB} is similar to the proof of Proposition \ref{prop:EB}, and we omit it. 
\begin{prop}\label{prop:EinverseB} We have:
    \begin{enumerate}[(i)]
    \item The union of $C_7$, $C_3$, $C_8$ and $C_4$ is a simple closed curve $C_{B^{-1}}$.
    \item $C_{B^{-1}}$ bounds the disk $E_{B^{-1}}\subset \partial  \mathbf{H}^2_{\mathbb C}$. 
    The interior of $E_{B^{-1}}$ is disjoint from  $\partial_{\infty}I_k^{+}\cup \partial_{\infty}I_k^{-}\cup\partial_{\infty}I_k^{*}$ for any $k\in \mathbb{Z}$.
    \item $E_{B^{-1}} \cap A^{k}(E_{B^{-1}})= \emptyset $ if $k \neq 0$. 
    \end{enumerate}	
\end{prop}
From the equations of $E_B$ and $E_{B^{-1}}$, we have
\begin{prop}\label{prop:EBEinverseB} 
The disk $E_{B}$ intersects the disk $E_{B^{-1}}$ exactly at the point $x_2$. 
\end{prop}

\begin{prop}\label{prop:unknot} There is an unknotted infinite line $\mathcal{L}$ in $\partial  \mathbf{H}^2_{\mathbb C}$ such that:
    \begin{enumerate}[(i)]
	\item  \label{unknot:1} 
    $\mathcal{L} \cap \big(\mathcal{S} \cup_{k \in \mathbb{Z}}A^{k}(E_B \cup E_{B^{-1}})\big)$  is exactly the set $\cup_{k \in \mathbb{Z}}A^{k}(x_2)$.
	\item \label{unknot:2} 
    The complement of $\mathcal{S} \cup_{k \in \mathbb{Z}}A^{k}(E_B \cup E_{B^{-1}})$
    in  $\partial\mathbf{H}^2_{\mathbb C}$ is an $A$-invariant $(\mathbb{R}^2-\mathbb{D}^2)\times\mathbb{R}$ topologically. 
    \end{enumerate}	
\end{prop}
\begin{proof}Let $\mathcal{L}$ be the
line 
\begin{equation*}
	\{[x,0]: x\in \mathbb{R}\} 
\end{equation*} 
in $\mathcal{N}$. 
Then Item \ref{unknot:1} 
can be checked directly.
 
Now, the complement of $\mathcal{S} \cup_{k \in \mathbb{Z}}A^{k}(E_B \cup E_{B^{-1}})$ in  $\partial  \mathbf{H}^2_{\mathbb C}$ 
is isotopic to the boundary of a small neighborhood of $\mathcal{L} \cup_{k \in \mathbb{Z}}A^{k}(E_B \cup E_{B^{-1}})$ in $A$-equivalent way,  
so it is a $A$-invariant $(\mathbb{R}^2-\mathbb{D}^2) \times  \mathbb{R}$  topologically.
\end{proof}
	
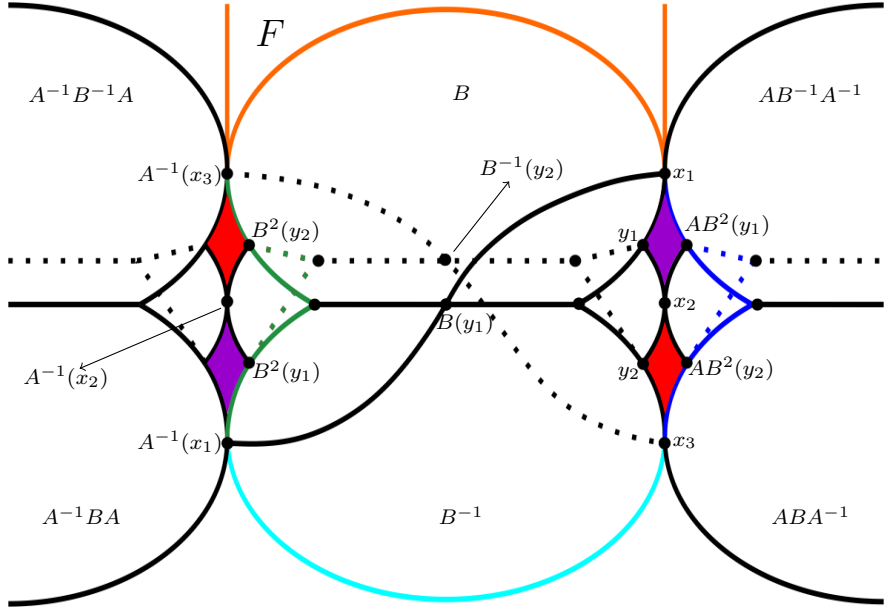
\begin{figure}[htbp]
	\begin{center}
		\begin{tikzpicture}
		\node at (0,0) {\includegraphics[width=11.5cm,height=8.0cm]{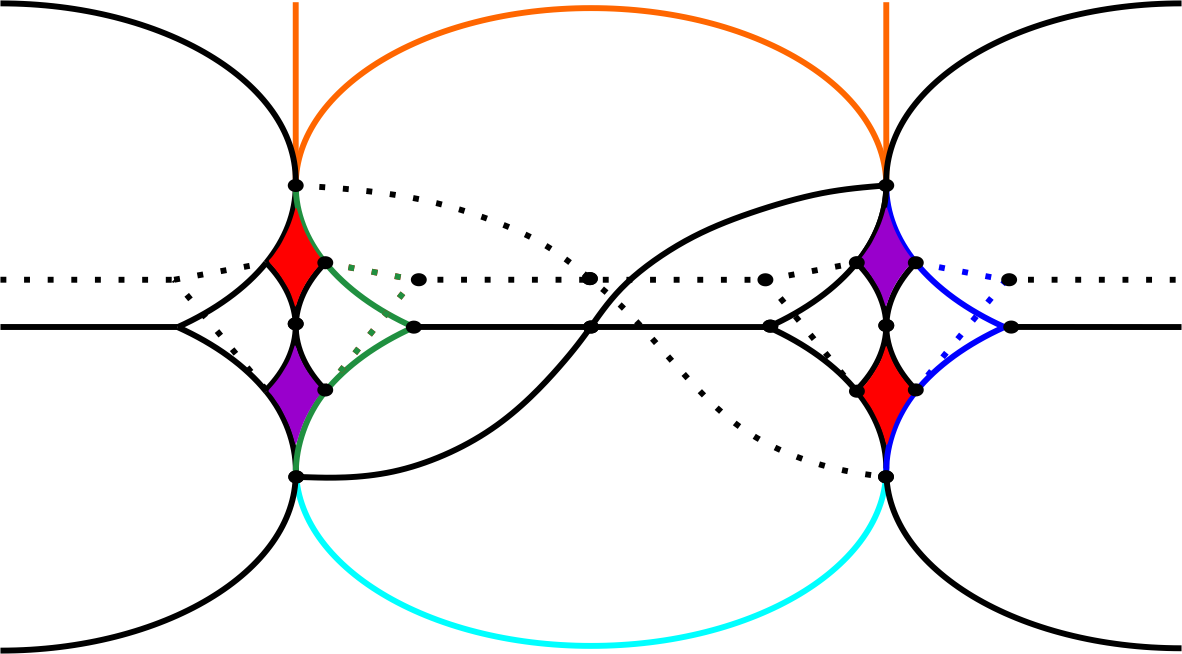}};

		\node at (0.2,2.8){\tiny $B$};
		\node at (0.2,-2.8){\tiny $B^{-1}$};
		
		\node at (4.8,-2.8){\tiny $ABA^{-1}$};
		\node at (4.8,2.8){\tiny $AB^{-1}A^{-1}$};
		
		\node at (-4.8,-2.8){\tiny $A^{-1}BA$};
		\node at (-4.8,2.8){\tiny $A^{-1}B^{-1}A$};

		\node at (-2.31,3.6){\large $F$};
		
		\node at (3.14,1.72){\tiny $x_1$};
		\node at (3.15,-0.0){\tiny $x_2$};
		\node at (3.14,-1.83){\tiny $x_3$};

		
		\node at (-3.5,1.72){\tiny $A^{-1}(x_3)$};
		\node at (-5,-0.97){\tiny $A^{-1}(x_2)$};
            \draw[->] (-2.96,-0.05)--(-4.82,-0.85);
		\node at (-3.5,-1.83){\tiny $A^{-1}(x_1)$};
		
      

		
		\node at (2.41,0.91){\tiny $y_1$};
		\node at (2.41,-0.92){\tiny $y_2$};
		\node at (-2.11,0.98){\tiny $B^2(y_2)$};
		\node at (-2.1,-0.9){\tiny $B^2(y_1)$};
		\node at (3.7,1.05){\tiny $AB^2(y_1)$};
		\node at (3.75,-0.85){\tiny $AB^2(y_2)$};
		\node at (1,1.8){\tiny $B^{-1}(y_2)$};
        \draw[->] (0.1,0.7)--(0.83,1.63);
		\node at (0.25,-0.25){\tiny $B(y_1)$};
		
		\end{tikzpicture}
	\end{center}
	\caption{The topology of the Ford domain of $\Delta_{4,\infty,\infty;\infty}$. After capping off the holes in Figure \ref{figure:4pppford} by red and purples disks in Figure \ref{figure:4pppfordtop}, we need two once punctured disks $D_{-}$ (with green boundary) and $D_{+}$ (with blue boundary), and a triangle $F$ (with  orange boundary) in Figure \ref{figure:4pppfordtop}. Since there are too many vertices, we do not label vertices $f(w_i)$ for $f \in \Sigma$ and $i=1,2,3,4$. They are the same as in Figure \ref{figure:4pppfordabstract}.}
	\label{figure:4pppfordtop}
\end{figure}

\subsection{Three more disks and seventeen more arcs}\label{subsection:moredisks}
More disks are used to cut out a 3-ball in $\partial_{\infty} D_{\Sigma}$. 
The quotient space of this 3-ball by suitable side-parings is the 3-manifold at infinity in Theorem \ref{thm:4pp}.

\begin{itemize}
\item Let $C_{17}$ be the $\mathbb{C}$-arc in $s^{+}_0 \cap \partial \mathbf{H}^2_{\mathbb C}$ joining 
$x_1$ and $B^{-1}(x_3)$ given by 
\begin{equation*}
    C_{17}=\left\{\begin{bmatrix}
-\frac{3}{2}-i\sqrt{2}e^{is} \\ 
i-\sqrt{2}e^{is} \\ 
1\end{bmatrix}
: \ s\in [\frac{3\pi}{4},\frac{9\pi}{4}]\right\}.
\end{equation*}
\item Let $C_{19}$  be the $\mathbb{C}$-arc in $\partial \mathbf{H}^2_{\mathbb C}$ joining 
$q_{\infty}$ and $B^{-1}(x_3)$ given by 
\begin{equation*}
    C_{19}=\left\{\begin{bmatrix}
-\frac{1}{2}+si \\ 
1 \\ 
1\end{bmatrix}
: \ s\in [2,\infty]\right\}.
\end{equation*}
\item Let $C_{20}$  be the $\mathbb{C}$-arc in $\partial \mathbf{H}^2_{\mathbb C}$ joining 
$x_1$ and $q_{\infty}$ given by 
\begin{equation*}
    C_{20}=\left\{\begin{bmatrix}
-\frac{1}{2}-si \\ 
-1 \\ 
1\end{bmatrix}
: \ s\in [2,\infty]\right\}.
\end{equation*}
\end{itemize}

Now, we construct a disk $F$.
First, take vertical upward rays (part of $\mathbb{C}$-arcs) along the first half of $C_{17}$, with paramatrization
$$[-\sqrt{2}\cos{(s)}, 1-\sqrt{2}\sin(s), -2\sqrt{2}\cos{(s)}+a]$$
in Heisenberg coordinates with 
$s\in [\frac{3\pi}{4},\frac{3\pi}{2}]$ and $ a\in [0,+\infty]$. 
Denote this part by $F_1$.
Next, take vertical downward rays (part of the $\mathbb{C}$-arcs) along the second half of $C_{17}$,
 with paramatrization 
$$[-\sqrt{2}\cos{(s)}, 1-\sqrt{2}\sin(s), -2\sqrt{2}\cos{(s)}+a]$$ in Heisenberg coordinates with $s\in [\frac{3\pi}{2},\frac{9\pi}{4}]$ and $ a\in [0,-\infty]$. 
 Denote this part by $F_2$. 
Finally, take a vertical half-plane $F_3$, with paramatrization 
$$[0, (1+\sqrt{2})s, t]$$ in Heisenberg coordinates with $s\in[1+\infty]$ and $t\in [-\infty,+\infty]$. 
Let $$C_F=C_{17}\cup C_{19}\cup C_{20} \quad \text{and}  \quad F=F_1\cup F_2\cup F_3.$$
They are shown in Figure \ref{figure:F123}.
We draw a part of $F$ together with several spinal spheres in Figure \ref{figure:disk_F}.
Using the previous method similarly, we give the following results.
\begin{figure}[tbp]
\centering
\begin{minipage}{.5\textwidth}
  \centering
  {\includegraphics[width=6.5cm,height=5.5cm]
  {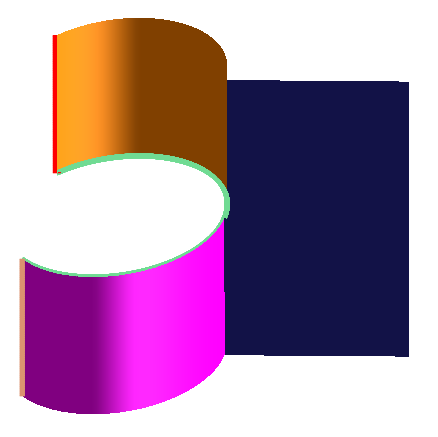}}
  \caption{Disks $F_1$ (coral),  $F_2$ (magenta)
    and $F_3$ (blue), curves $C_{17}$ (green), $C_{19}$ (red) and  $C_{20}$ (tan).}
  \label{figure:F123}
\end{minipage}%
\hfill
\begin{minipage}{.5\textwidth}
  \centering
  {\includegraphics[width=7cm,height=5.8cm]
  {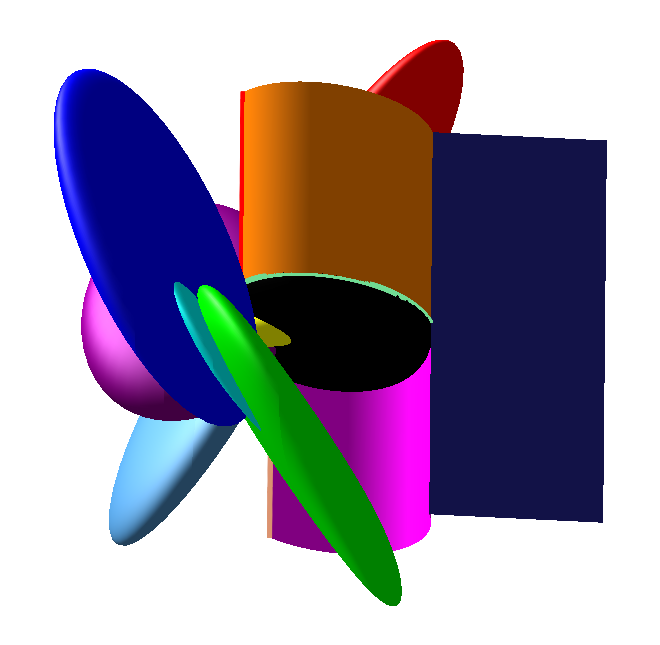}}
  \caption{Add part of the disk $F$ to Figure \ref{figure:4pppford}.}
  \label{figure:disk_F}
\end{minipage}
\end{figure}

\begin{prop}\label{prop:diskF} We have:
    \begin{enumerate}[(i)]
	\item \label{diskF1} The union $C_F$ is a simple closed curve.
	\item \label{diskF2} $C_F$ bounds the disk $F \subset \partial\mathbf{H}^2_{\mathbb C}$.
    The interior of $F$ has no intersection with $\partial_{\infty}I_k^{+}\cup \partial_{\infty}I_k^{-}\cup\partial_{\infty}I_k^{*}$ for any $k \in \mathbb{Z}$. 
    \end{enumerate}	
\end{prop}
\begin{proof} 
For Item \ref{diskF1}, by the equations of the arcs of $C_F$, it is easy to verify that any pair intersect only at one end point.

For Item \ref{diskF2}, according to the structure of $F$, we naturally have $\partial F=C_F$.
First, the projection $\Pi(F_3)$ is a ray in $\mathbb{C}$, satisfying $$\Pi\big(F_3\cap(\partial_{\infty}I_k^{+}\cup \partial_{\infty}I_k^{-}\cup\partial_{\infty}I_k^{*})\big)=(0,1+\sqrt{2})\in\Pi(\partial_{\infty}I_0^{+}).$$
Second, 
$$\Pi(F_2)=\big(-\sqrt{2}\cos{(s)},1-\sqrt{2}\sin{(s)}\big)\subset \partial\Pi(\partial_{\infty}I_0^{+}),$$ 
where $s\in [\frac{3\pi}{4},\frac{3\pi}{2}]$. 
Therefore, the interior of $F_2$ is disjoint from $\partial_{\infty}I^{+}_k(k\neq 1)$, $\partial_{\infty}I^{-}_k$, and $\partial_{\infty}I^{\star}_k(k\neq 1)$ for $k \in \mathbb{Z}$.
We need only to check that if $F_2\cap(\partial_{\infty}I^{+}_1\cup \partial_{\infty}I^{\star}_1)$ is an empty set.

On the one hand, consider the plane
$$3(X + 1 + Y) - 2-T=0$$ in the Heisenberg group.
For any point $(x_0,y_0,t)\in F_2$, where $(x_0,y_0,t_0)\in C_{17}^1$ and $t\leq t_0$,  we have 
\begin{equation*}
    \begin{aligned}
        3(x_0 + 1 + y_0) - 2-t&\ge 3(x_0 + 1 + y_0) - 2-t_0\\
        &=3(-\sqrt{2}\cos{(s)}+1+1-\sqrt{2}\sin(s))-2+2\sqrt{2}\cos{(s)}\\
        &=4-\sqrt{2}\cos{(s)}-3\sqrt{2}\sin(s)\ge 0. 
    \end{aligned}
\end{equation*}
This means that $F_2$ is on one side of this plane.

On the other hand, for any point $(x_0,y_0,t_0)\in s^{+}_1$,
\begin{equation*}
    \begin{aligned}
        3(x_0 + 1 + y_0)-2-t_0&=\mp\sqrt{2\cos{\alpha}}\left(\cos{\left(\frac{\alpha}{2}-\beta\right)}+\sin{\left(\frac{\alpha}{2}-\beta\right)}\right)-2\sin{\alpha}-6\\
        &\le 2\sqrt{\cos{\alpha}}-2\sin{\alpha}-6<0.
    \end{aligned}
\end{equation*}
Moreover, 
$$\partial_{\infty}I_1^{*}=\left\{\begin{bmatrix}
-\frac{e^{-\alpha i}}{2}\pm2\sqrt{\cos{\alpha}}\cdot e^{(-\frac{\alpha}{2}+\beta)i}-2 \\ 
-2\pm\sqrt{\cos{\alpha}}\cdot e^{(-\frac{\alpha}{2}+\beta)i} \\ 
1\end{bmatrix}
: \ \alpha\in [-\frac{\pi}{2},\frac{\pi}{2}],\ \beta\in [0,\pi)\right\}.
$$
So for any point $(x_0,y_0,t_0)\in \partial_{\infty}I^{*}_1$, 
\begin{equation*}
    \begin{aligned}
    3(x_0 + 1 + y_0)-2-t_0
    &=-5\pm\sqrt{\cos{\alpha}}\left(3\cos{\left(\frac{\alpha}{2}-\beta\right)}+\sin{\left(\frac{\alpha}{2}\right)}-\beta\right)-\sin{\alpha}<0.
    \end{aligned}
\end{equation*}
Therefore, $F_2$ and $(\partial_{\infty}I^{+}_1\cup \partial_{\infty}I^{\star}_1)$ are separated by the above plane.
\end{proof}

Table \ref{table_curves} gives the other curves we used.

Consider a curve $C_{+}$ in $\partial_{\infty} D_{\Sigma}-\big(\cup_{k \in \mathbb{Z}} A^{k}(E_B \cup E_{B^{-1}})\big)$ which consisting of eight edges:
 
\begin{itemize}[-]
    \item glue two copies of $C_6$ together along $x_1$;
    \item take a copy $C_{21}$, $C_{22}$, $C_{23}$, and $C_{24}$;
    \item glue two copies of $C_8$ together along $x_3$. 
\end{itemize} 
 These edges were then glued together along $AB^2(y_1)$, $A(w_3)$, $AB^2(y_2)$, and $A(w_4)$ to obtain the simple closed curve $C_{+}$ (the union of  blue edges in Figure \ref{figure:4pppfordtop}).  
 Similarly, there is a curve $C_{-}=A^{-1}(C_{+})$ that is the union of green edges in Figure \ref{figure:4pppfordtop}.

\begin{prop}\label{prop:diskDminus} There are  proper once punctured  disks $D_{-}$  and $D_{+}$  in the complement of the union  of $\mathcal{S} \cup_{k \in \mathbb{Z}}A^{k}(E_B \cup E_{B^{-1}})$ in  $\partial  \mathbf{H}^2_{\mathbb C}$ such that:
\begin{enumerate}[(i)]
    \item $\partial D_{-}$ is $C_{-}$.
    \item $\partial D_{+}$ is $C_{+}$.
    \item {$D_{+}=A(D_{-})$}.
    \item $D_{-}\cap F=C_{19}$.
    \item $D_{+}\cap F=C_{20}$.
    \item $D_{+}\cap D_{-}=\emptyset$.  
\end{enumerate}	
\end{prop}
 
\begin{proof} 
With the boundary $\partial D_{-}$ as stated in Proposition \ref{prop:diskDminus}, then by Proposition \ref{prop:unknot},  $\partial D_{-}$ bounds a proper once punctured disk $D_{-}$ in $\partial  \mathbf{H}^2_{\mathbb C}$, the punctured point is $q _{\infty}$. 
Let $D_{+}=A(D_{-})$. 
Since $\partial D_{-} \cap \partial D_{+}=\emptyset$, by the standard argument in 3-manifold theory, we may assume $D_{-} \cap D_{+}= \emptyset$. 
Moreover, we may isotope $D_{-}$ such that $D_{-}\cap F=C_{19}$ and $D_{+}\cap F=C_{20}$.
We may also assume that the boundaries of $D_{-}$ and $D_{+}$  are fixed during the isotopy. 
\end{proof}

\begin{table}[tbp]
    \caption{Some curve information.}
    \begin{tabular}{|c|c|c|}
    \hline 
    name&ends&belongs to
    \\
    \hline  
    $C_{18}=B(C_{17})$&$x_3, B(x_1)$
    &$s^{-}_0 \cap \partial \mathbf{H}^2_{\mathbb C}$\\
    
    \hline
    $C_9=A^{-1}(C_2)$&$A^{-1}(x_2), B^2(y_1)$
    &$s^{*}_0 \cap \partial \mathbf{H}^2_{\mathbb C}$\\
    
    \hline
    $C_{10}=A^{-1}(C_4)$&$A^{-1}(x_2), B^2(y_2)$
    &$s^{*}_0 \cap \partial \mathbf{H}^2_{\mathbb C}$\\
    
    \hline
    $C_{11}=A^{-1}(C_6)$&$A^{-1}(x_1), B^2(y_1)$
    &$s^{-}_0 \cap \partial \mathbf{H}^2_{\mathbb C}$\\
    
    \hline
    $C_{12}=A^{-1}(C_8)$&$A^{-1}(x_3), B^2(y_2)$
    &$s^{+}_0 \cap \partial \mathbf{H}^2_{\mathbb C}$\\
    
    \hline
    $C_{13}=B(C_5)$&$B(x_1), B(y_1)$
    &$s^{-}_0 \cap \partial \mathbf{H}^2_{\mathbb C}$\\
    
    \hline
    $C_{14}=B^{-1}A^{-1}(C_6)$&$x_1, B(y_1)$
    &$s^{+}_0 \cap \partial \mathbf{H}^2_{\mathbb C}$\\
    
    \hline
    $C_{15}=B^{-1}(C_7)$&$B^{-1}(x_3), B^{-1}(y_3)$
    &$s^{+}_0 \cap \partial \mathbf{H}^2_{\mathbb C}$\\
    
    \hline
    $C_{16}=BA^{-1}(C_8)$&$x_3, B(y_1)$
    &$s^{-}_0 \cap \partial \mathbf{H}^2_{\mathbb C}$\\
    
    \hline
    $C_{21}$& $AB^2(y_1), A(w_3)$
    &$s^{-}_1 \cap s^{*}_1 \cap\partial \mathbf{H}^2_{\mathbb C}$\\
    
    \hline
    $C_{22}$& $AB^2(y_1), A(w_4)$
    &$s^{-}_0 \cap s^{*}_1 \cap \partial \mathbf{H}^2_{\mathbb C}$\\
    
    \hline
    $C_{23}$& $AB^2(y_2), A(w_3)$
    &$s^{+}_1 \cap s^{*}_1 \cap \partial \mathbf{H}^2_{\mathbb C}$\\
    
    \hline
    $C_{24}$& $AB^2(y_2), A(w_4)$
    &$s^{+}_1 \cap s^{*}_1 \cap \partial \mathbf{H}^2_{\mathbb C}$\\
    
    \hline
    $C_{25}=B^{-1}A^{-1}(C_{22})$& $w_1, B(y_1)$
    &$s^{+}_0 \cap s^{-}_0 \cap \partial \mathbf{H}^2_{\mathbb C}$\\
    
    \hline
    $C_{26}=B^{-1}A^{-1}(C_{21})$& $w_4, B(y_1)$
    &$s^{+}_0 \cap s^{-}_0 \cap \partial \mathbf{H}^2_{\mathbb C}$\\
    
    \hline
    $C_{27}=BA^{-1}(C_{24})$& $w_3, B^{-1}(y_3)$
    &$s^{+}_0 \cap s^{-}_0 \cap \partial \mathbf{H}^2_{\mathbb C}$\\
    
    \hline
    $C_{28}=BA^{-1}(C_{23})$& $w_2, B^{-1}(y_3)$
    &$s^{+}_0 \cap s^{-}_0 \cap \partial \mathbf{H}^2_{\mathbb C}$\\
    
    \hline
    \end{tabular}
    \label{table_curves}
\end{table}


 
\section{3-manifold at infinity of
\texorpdfstring{$\Delta_{4,\infty,\infty;\infty}$}{the complex hyperbolic group} %
} 
\label{section:3mfd4ppp}

In this section, we prove Theorem \ref{thm:4pp}.

In Section \ref{section:topology4pppford}, we prove that $\partial_{\infty} D_{\Sigma}-(\cup_{k \in \mathbb{Z}} A^{k}(E_B \cup E_{B^{-1}}))$ is a topologically $A$-invariant $(\mathbb{R}^2-\mathbb{D}^2) \times \mathbb{R}$.
Recall that $A(D_{-})=D_{+}$, the once punctured disks $D_{-}$ and $D_{+}$ separate $\partial_{\infty} D_{\Sigma}-( \cup_{k \in \mathbb{Z}} A^{k} (E_B \cup E_{B^{-1}}))$ into three parts, and exactly one of them, denoted as $U$, contains both $D_{-}$ and $D_{+}$. 
Now $U$ is homeomorphic to $(\mathbb{R}^2-\mathbb{D}^2) \times  [0,1]$, which serves as a fundamental domain of the $A$-action. 
Moreover, the disk $F$ in Subsection \ref{subsection:moredisks} cuts $U$ into a 3-ball, denoted by $N$.   

Figure \ref{figure:3ballfor4ppp} gives a 2-cell decomposition of the boundary 2-sphere of $N$.  
The curves $C_{-}$ and $C_{+}$, which are
the union of  green and blue edges in Figure \ref{figure:4pppfordtop}, now correspond to the green and blue paths in Figure \ref{figure:3ballfor4ppp}. 
In addition, from our choice of $D_{-}$ and $D_{+}$, 
$U$ only contains half of $s^{\star}_0 \cap \partial \mathbf{H}^2_{\mathbb C}$, and 
half of $s^{\star}_1 \cap \partial \mathbf{H}^2_{\mathbb C}$. 
The side-pairing between them is $AB^2$. 

We list several cells in Table \ref{table_cells}.
The 2-cell decomposition of $\partial N$ is as follows: 
\begin{itemize} [-]
	\item we divide $s^{+}_0 \cap \partial\hc$ into four subdisks $B_1$, $B_2$, $B_3$ and  $B_4$, and  divide $s^{-}_0 \cap \partial\hc$  into four subdisks $B_1^{-1}$, $B_2^{-1}$, $B_3^{-1}$ and  $B_4^{-1}$;
	\item we divide $s^{*}_0 \cap \partial\hc$ into disks $B_{+}^2$ and $B_{-}^2$, and divide $s^{*}_1 \cap \partial\hc$ into two disks $AB_{+}^2A^{-1}$ and $AB_{-}^2A^{-1}$;
	\item two copies of $F$, say $F_+$ and $F_-$;
	\item two copies of $E_{B}$, say $E^{+}_B$ and  $E^{-}_B$;
	\item two copies of $E_{B^{-1}}$, say $E_{B^{-1}}^{+}$ and  $E_{B^{-1}}^{-}$;  
	\item we divide $D_{-}$ into two disks $D_{-,1}$ and $D_{-,2}$, and divide $D_{+}$ into two disks  $D_{+,1}$ and $D_{+,2}$. 
\end{itemize}

\begin{figure}[tbp]
	\begin{center}
		\begin{tikzpicture}
		\node at (0,0) {\includegraphics[width=12cm,height=10cm]{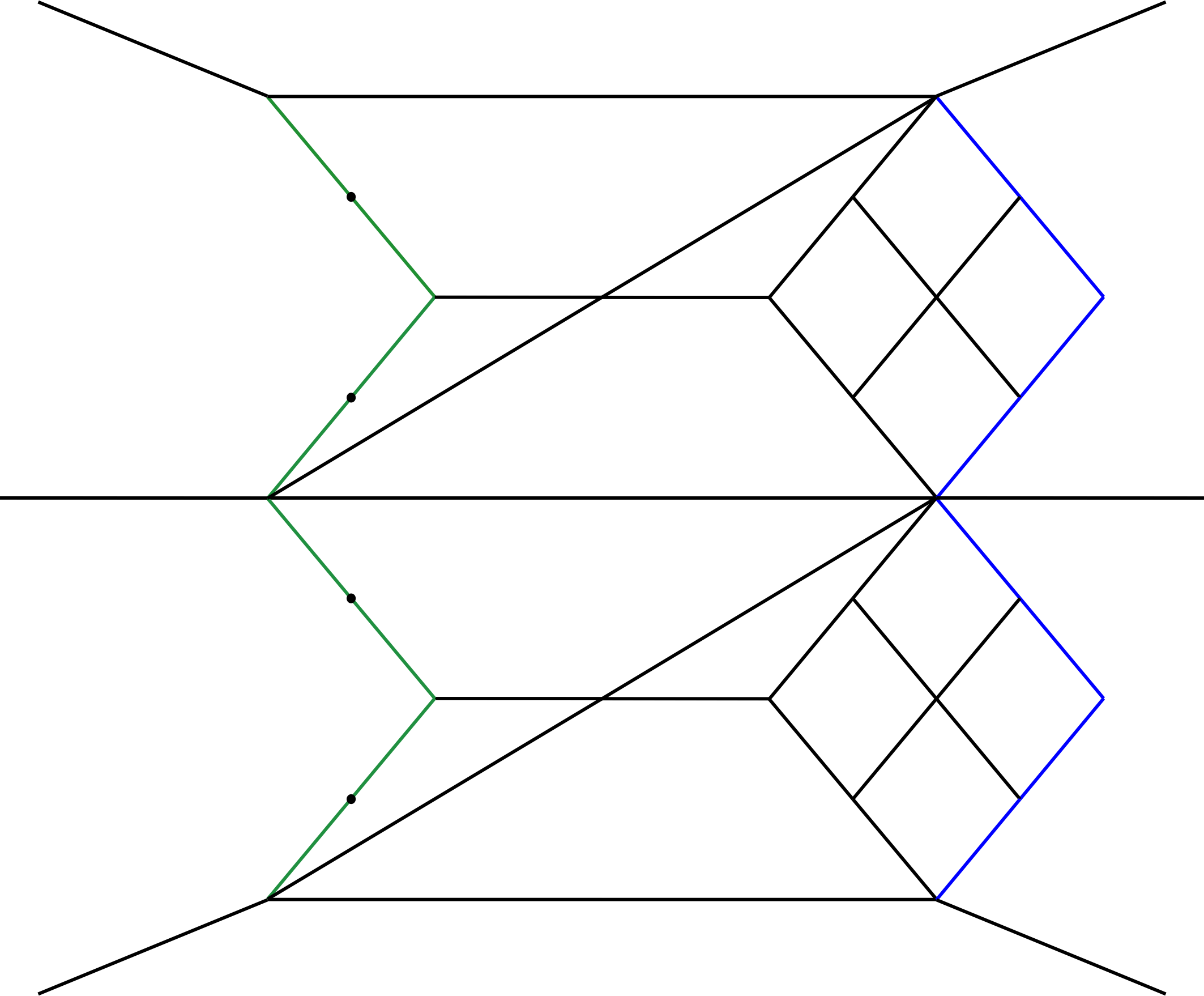}};
		
		\node at (0.2,4.3){\large $F_{-}$};
		\node at (0.6,-4.4){\large $F_{+}$};
            \node at (-4.1,2.1){\large $D_{2,-}$};
		\node at (-4.1,-2.1){\large $D_{1,-}$};
            \node at (5.8,4.1){\large $D_{1,+}$};
		\node at (5.8,-4.23){\large $D_{2,+}$};
  
            \node at (-1.2,3.3){\large $B_2$};
            \node at (-1.2,-1){\large $B_2^{-}$};
		\node at (1.4,2.4){\large $B_4$};
            \node at (-1.45,1.57){\large $B_4^{-}$};

            \node at (1.1,0.8){\large $B_1^{-}$};
            \node at (1.1,-3.15){\large $B_1$};
            \node at (1.4,-1.6){\large $B_3^{-}$};
            \node at (-1.6,-2.57){\large $B_3$};

            \node at (3.3,3.05){\tiny $E_B^{+}$};
            \node at (3.3,1){\tiny $E_{B^{-1}}^{+}$};
            \node at (3.3,-3.15){\tiny $E_B^{-}$};
            \node at (3.3,-1){\tiny $E_{B^{-1}}^{-}$};

            \node at (2.5,2.05){\tiny $B_{+}^2$};
            \node at (2.5,-2.05){\tiny $B_{-}^2$};
            \node at (4.1,2.05){\tiny $AB_{+}^2A^{-1}$};
            \node at (4.1,-2.05){\tiny $AB_{-}^2A^{-1}$};
  
            \node at (-2.8,4.2){\tiny $A^{-1}(x_3)$};
            \node at (-3.8,0.25){\tiny $A^{-1}(x_1)$};
            \node at (-2.9,-4.3){\tiny $A^{-1}(x_3)$};

            \node at (-1.95,2.05){\tiny $w_4$};
            \node at (-1.95,-2.05){\tiny $w_3$};
            \node at (1.6,1.78){\tiny $w_1$};
            \node at (1.6,-2.2){\tiny $w_2$};
            
            \node at (0.33,-2.25){\tiny $B^{-1}(y_2)$};
            \node at (0.33,1.78){\tiny $B(y_1)$};

            \node at (3.3,4.2){\tiny $x_1$};
            \node at (3.3,2.3){\tiny $x_2$};
            \node at (3.3,0.3){\tiny $x_3$};
            \node at (3.3,-4.2){\tiny $x_1$};
            \node at (3.35,-2.3){\tiny $x_2$};

            \node at (5.5,2.05){\tiny $A(w_3)$};
            \node at (5.5,-2.05){\tiny $A(w_4)$};
            
            \node at (4.68,3.15){\tiny $AB^2(y_1)$};
            \node at (4.75,1){\tiny $AB^2(y_2)$};
            \node at (4.65,-3.15){\tiny $AB^2(y_1)$};
            \node at (4.78,-1){\tiny $AB^2(y_2)$};

            \node at (-3.05,3.0){\tiny $B^2(y_2)$};
            \node at (-3.01,1.1){\tiny $B^2(y_1)$};
            \node at (-3.05,-3.0){\tiny $B^2(y_2)$};
            \node at (-3.01,-1.2){\tiny $B^2(y_1)$};

            \node at (2.33,3.15){\tiny $y_1$};
            \node at (2.3,1){\tiny $y_2$};
            \node at (2.33,-3.15){\tiny $y_1$};
            \node at (2.3,-1){\tiny $y_2$};
		
		\end{tikzpicture}
	\end{center}
	\caption{The boundary of the polytope $N$ for  $\Delta_{4,\infty,\infty;\infty}$.}
	\label{figure:3ballfor4ppp}
\end{figure}
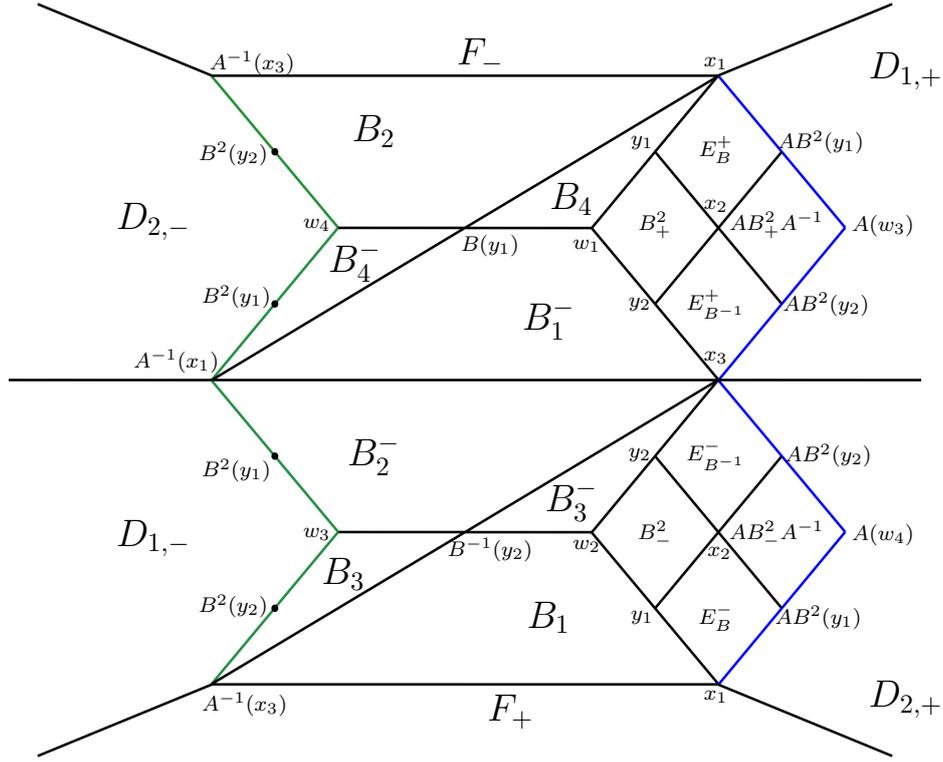

By matching the edges of 2-cells in Table \ref{table_cells}, we have that
$AB^2(B^2_{-})=AB^2_{-}A^{-1}$,  $AB^2(B^2_{+})=AB^2_{+}A^{-1}$ and $B(B_i)=B^{-1}_i$ for $i=1,2,3,4$. We also  have that $A(D_{-,1})=D_{+,1}$ and $A(D_{-,2})=D_{+,2}$.

\begin{table}[htbp]
    \caption{2-cells and their edges.}
    \begin{tabular}{|c|c|c|c|}
    \hline 
    name&edges&name&edges
    \\
    \hline  
    $B_1$&$C_{17}, C_{15}, C_{28}, B^2A^{-1}(C_{22}), C_1$
    &$B_1^{-1}$&$C_{18}, C_{7}, B^2A^{-1}(C_{23}), C_{25}, C_{13}$\\
    
    \hline
    $B_2$&$C_{17}, C_{12}, A^{-1}(C_{24}), C_{26}, C_{14}$
    &$B_2^{-1}$&$C_{18}, C_{16}, C_{27}, A^{-1}(C_{21}), C_{11}$\\

    \hline
    $B_3$&$C_{15}, C_{12}, A^{-1}(C_{23}), C_{27}$
    &$B_3^{-1}$&$C_{7}, C_{16}, C_{28}, B^2A^{-1}(C_{24})$\\

    \hline
    $B_4$&$C_{25}, B^2A^{-1}(C_{21}), C_{5}, C_{14}$
    &$B_3^{-1}$&$A^{-1}(C_{22}), C_{26},  C_{13}, C_{11}$\\

    \hline
    $B_{+}^2$&$ B^2A^{-1}(C_{21}), C_1, C_3, B^2A^{-1}(C_{23})$
    &$AB_{+}^2A^{-1}$&$ C_{21}, C_2, C_4, C_{23}$\\

    \hline
    $B_{-}^2$&$B^2A^{-1}(C_{22}), C_1, C_3, B^2A^{-1}(C_{24})$
    &$AB_{-}^2A^{-1}$&$ C_{22}, C_2, C_4, C_{24}$\\
    
    \hline
    \end{tabular}
    \label{table_cells}
\end{table}


Let $M$ be the 3-manifold at infinity of the even subgroup $\Sigma$ of $\Delta_{4,\infty,\infty;\infty}$, which is the quotient space of $N$. 
According to the above analysis, the side-pairings on $N$ are
\begin{equation*}
    \begin{aligned}
        f_1=B&: B_{1}\longrightarrow B^{-1}_{1},\\
        f_2=B&: B_{2}\longrightarrow B^{-1}_{2},\\
        f_3=B&: B_{3}\longrightarrow B^{-1}_{3},\\
        f_4=B&: B_{4}\longrightarrow B^{-1}_{4},\\
	f_5=AB^2&: B_{-}^2\longrightarrow AB_{-}^2A^{-1},\\
        f_6=AB^2&: B_{+}^2\longrightarrow AB_{+}^2A^{-1},\\
	f_7=A&: D_{1,-}\longrightarrow D_{1,+},\\
	f_8=A&: D_{2,-}\longrightarrow D_{2,+},\\
	f_9&: F_{-}\longrightarrow F_{+},\\
	f_{10}&: E_{B}^{-}\longrightarrow E_{B}^{+},\\
	f_{11}&: E_{B^{-1}}^{-}\longrightarrow E_{B^{-1}}^{+}.
    \end{aligned}
\end{equation*}
	

\begin{table}[!htbp]
	\caption{Edge cycles and relations for the 3-manifold at infinity  of $\Delta_{4,\infty,\infty;\infty}$.}
	\centering
	\begin{tabular}{c|c|c}
		\toprule
		\textbf{edge} & \textbf{edge Cycle}  & \textbf{Cycle relation}\\
		\midrule
		$e_{1}$ & $e_1\xrightarrow{f_9} e_{41}  
            \xrightarrow{f_7} e_{21} 
            \xrightarrow{f_8^{-1}}e_1$ 
            & $f^{-1}_8f_7f_9$ \\  [1 ex]
		
    $e_{2}$ & $e_2\xrightarrow{f_2} e_{35} 
    \xrightarrow{f_3^{-1}} e_{42} 
    \xrightarrow{f_7} e_{19} 
    \xrightarrow{f_{11}^{-1}}e_{22}
    \xrightarrow{f_8^{-1}}e_{2}$ 
    & $f^{-1}_8f^{-1}_{11}f_7f^{-1}_3f_2$ \\  [1 ex]

    $e_{3}$ & $e_3\xrightarrow{f_2} e_{38} 
    \xrightarrow{f_3} e_{33} 
    \xrightarrow{f_5} e_{23} 
    \xrightarrow{f_8^{-1}}e_{3}$ 
    & $f^{-1}_8f_5f_3f_2$ \\  [1 ex]

    $e_{4}$ & $e_4\xrightarrow{f_2} e_{44} \xrightarrow{f_7} e_{12} \xrightarrow{f_6^{-1}}e_{7}
    \xrightarrow{f_4}e_{4}$  
    & $ f_4f_6^{-1}f_7f_2$ \\  [1 ex]

    $e_{5}$ & $e_5\xrightarrow{f_4} e_{48} \xrightarrow{f_8} e_{26} \xrightarrow{f_5^{-1}} e_{31} \xrightarrow{f_1}e_{5}$  
    & $f_1f_5^{-1}f_8f_4$ \\  [1 ex]

    $e_{6}$ & $e_6\xrightarrow{f_2} e_{45} 
    \xrightarrow{f_7} e_{11}
    \xrightarrow{f_{10}^{-1}} e_{27}  
    \xrightarrow{f_8^{-1}}e_{47}
    \xrightarrow{f_4^{-1}}e_{6}$
    & $f_4^{-1}f_8^{-1}f_{10}^{-1}f_7f_2$ \\  [1 ex]

    $e_{8}$ & $e_8\xrightarrow{f_4} e_{49} 
    \xrightarrow{f_1^{-1}} e_{29}
    \xrightarrow{f_{10}}e_{8}$
    & $ f_{10}f_1^{-1}f_4$ \\  [1 ex]

    $e_{9}$ & $e_9\xrightarrow{f_2} e_{36} 
    \xrightarrow{f_1^{-1}} e_{40} 
    \xrightarrow{f_9^{-1}}e_{9}$
     & $f_9^{-1}f_1^{-1}f_2$\\  [1 ex]

    $e_{10}$ & $e_{10}\xrightarrow{f_9} e_{28} 
    \xrightarrow{f_8^{-1}} e_{46} 
    \xrightarrow{f_7} e_{10}$
    & $f_7f_8^{-1}f_9$\\  [1 ex]

    $e_{13}$ & $e_{13}\xrightarrow{f_{10}^{-1}}e_{25}
    \xrightarrow{f_5^{-1}}e_{30}
    \xrightarrow{f_{10}} e_{14} 
    \xrightarrow{f_6}e_{13}$
    & $ f_6f_{10}f_5^{-1}f_{10}^{-1}$\\  [1 ex]
    
    $e_{15}$ & $e_{15}\xrightarrow{f_6}e_{18}
    \xrightarrow{f_7^{-1}}e_{43}
    \xrightarrow{f_3} e_{37} 
    \xrightarrow{f_1}e_{15}$
    & $f_1f_3f_7^{-1}f_6$ \\  [1 ex]

    $e_{16}$ & $e_{16}\xrightarrow{f_6}e_{17}
    \xrightarrow{f_{11}^{-1}}e_{24}
    \xrightarrow{f_{5}^{-1}} e_{32} 
    \xrightarrow{f_{11}}e_{16}$
    & $f_{11}f_5^{-1}f_{11}^{-1}f_6$\\  [1 ex]

    $e_{20}$ & $e_{20}\xrightarrow{f_{11}^{-1}}e_{34}
    \xrightarrow{f_3^{-1}}e_{39}
    \xrightarrow{f_1} e_{20}$
    & $f_1f_3^{-1}f_{11}^{-1}$\\
		\bottomrule
	\end{tabular}
	\label{table:edgecircle4pp}
\end{table}

By using the side-pairings on $\partial N$, the edge cycles for the 3-manifold $M$ is given in Table \ref{table:edgecircle4pp}.
Then we get a presentation of the fundamental group of the 3-manifold $M$. 
It is a group $\pi_{1}(M)$ with eleven generators $f_1, f_2, \cdots, f_{11}$ and thirteen relations. 
Via Magma, we can simplify to get a presentation as
$$\pi_1(M)=\left\langle x_1, x_2, x_3, x_4 \bigg| \begin{array}{c}   x_4 x^{-1}_2x_1x^{-1}_4x^{-1}_1x_2,\quad
x_1x^{-1}_4x_3x^{-1}_1x_4x^{-1}_3,\\ [3 pt]
 x_4x_2x_1x_3x_2x^{-1}_4x^{-2}_2x^{-1}_3x^{-1}_1
\end{array}\right\rangle,$$
where $x_1=f_1$, $x_2=f_2$, $x_3=f_3$ and $x_4=f_7$.

Let $8_1^4$ be the 4-component chain link    in Snappy  Census \cite{CullerDunfield:2014}, it is hyperbolic with volume 10.1494160641 and
$$\pi_{1}(S^3-8_1^4)=\big\langle a, b, c, d \big| \begin{array}{c}    ab^{-1}a^{-1}b,
acb^{-1}da^{-1}c^{-1}bd^{-1},
acd^{-1}c^2a^{-1}c^{-1}dc^{-2}
\end{array}\big\rangle.$$ 
Via Magma, there is an isomorphism 
$\phi:\pi_1(M) \rightarrow \pi_{1}(S^3-8_1^4)$ 
with 
 $$\phi(x_1)=c^{-1}b, \quad 
 \phi(x_2)=c^{-1},  \quad 
 \phi(x_3)=b^{-1}d,  \quad
 \phi(x_4)=a.$$
 The inverse isomorphism  is 
$\phi^{-1}:\pi_1(S^3-8_1^4)\rightarrow \pi_1(M)$ with 
$$\phi^{-1}(a)=x_4, \quad 
\phi^{-1}(b)=x_2^{-1}x_1,  \quad 
\phi^{-1}(c)=x_2^{-1},  \quad 
\phi^{-1}(d)=x_2^{-1}x_1x_3.$$

By the prime decompositions of 3-manifolds \cite{Hempel}, $M$ is the connected sum of $S^3-8_1^4$  with $L$, where
$L$ is a closed 3-manifold with a trivial fundamental group. By the solution of the
Poincar\'e conjecture,  $L$ is the 3-sphere. So $M$ is homeomorphic to $S^3-8_1^4$. This finishes the proof of Theorem \ref{thm:4pp}.
\begin{figure}[htbp]
	\begin{center}
		\begin{tikzpicture}
		\node at (0,0) {\includegraphics[width=12cm,height=10cm]{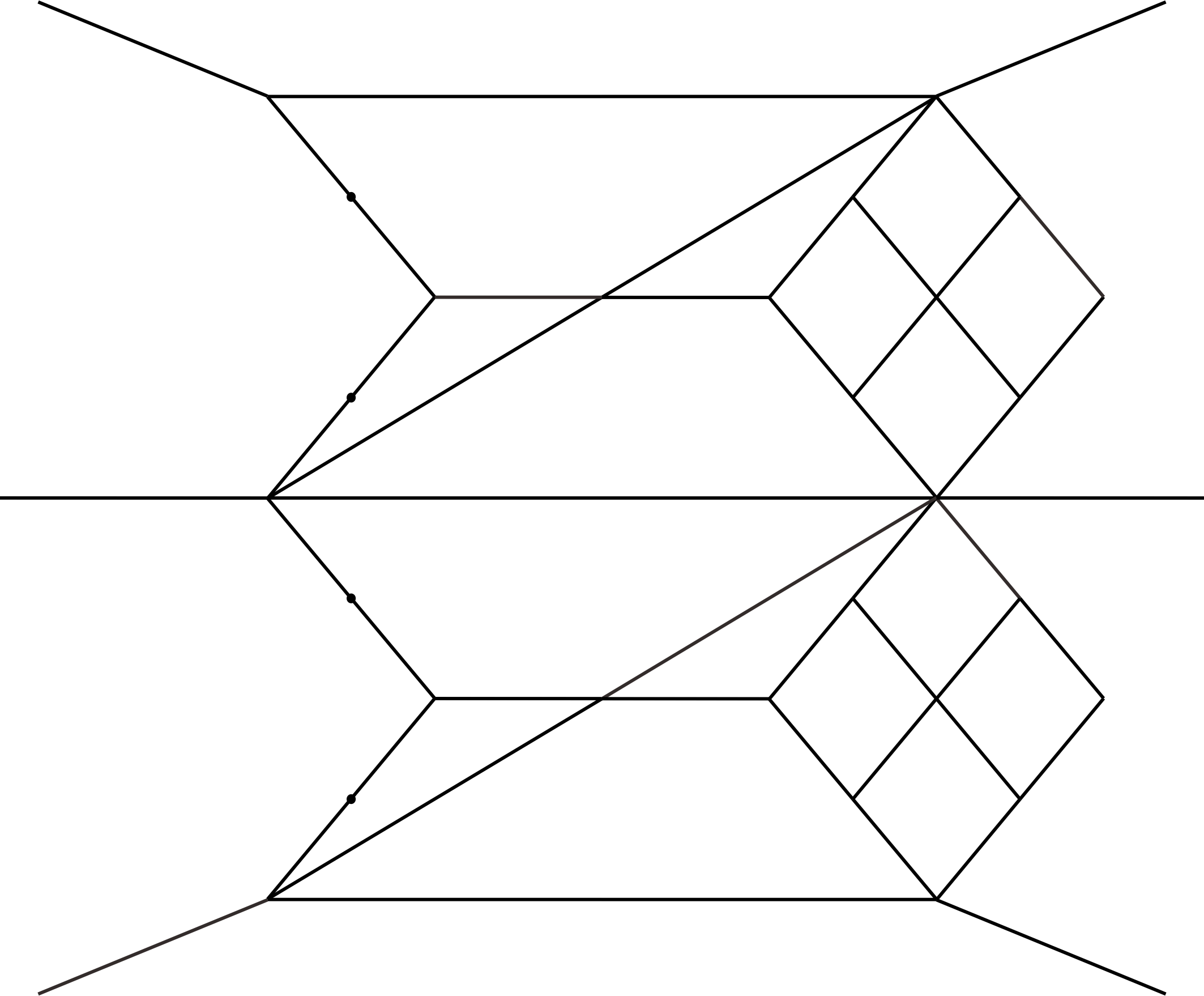}};
		
		\node at (-2.8,4.2){\tiny $A^{-1}(x_3)$};
            \node at (-2.9,-4.3){\tiny $A^{-1}(x_3)$};


            
            
            \node at (3.3,4.2){\tiny $x_1$};
            \node at (3.3,-4.2){\tiny $x_1$};

            \node at (5.5,2.05){\tiny $A(w_3)$};
            \node at (5.5,-2.05){\tiny $A(w_4)$};
            
              \node at (-4.1,2.1){\large $D_{2,-}$};
            \node at (-4.1,-2.1){\large $D_{1,-}$};

		\node at (0,-0.13){$e_{36}$};
		\node at (0,4.2){ $e_{9}$};
		\node at (0,-4.2){ $e_{40}$};

            \node at (-4.7,-0.16){$e_{46}$};
		\node at (-4.5,4.65){ $e_{1}$};
		\node at (-4.5,-4.73){ $e_{41}$};
  \node at (4.7,-0.16){$e_{21}$};
		\node at (4.47,4.7){ $e_{10}$};
		\node at (4.51,-4.7){ $e_{28}$};

        \node at (-1,2.155){ $e_{4}$};
        \node at (-1,-1.85){ $e_{38}$};
        \node at (1.01,1.85){ $e_{5}$};
        \node at (1.01,-2.22){ $e_{37}$};

        \node at (-2.63,3.48){ $e_{2}$};
        \node at (-1.845,2.55){ $e_{3}$};
        \node at (4.0,3.58){ $e_{11}$};
        \node at (4.75,2.65){ $e_{12}$};

        \node at (4.03,-3.58){ $e_{27}$};
        \node at (4.8,-2.65){ $e_{26}$};
        \node at (-3.15,-3.3){ $e_{42}$};
        \node at (-2.45,-2.45){ $e_{43}$};
        
        \node at (-3.15,0.65){ $e_{47}$};
        \node at (-2.46,1.53){ $e_{48}$};
        \node at (4.12,0.51){ $e_{19}$};
        \node at (4.93,1.53){ $e_{18}$};
        \node at (-3.16,-0.65){ $e_{45}$};
        \node at (-2.46,-1.53){ $e_{44}$};

        \node at (3.11,-0.69){ $e_{34}$};
        \node at (1.75,-1.5){ $e_{33}$};
        \node at (2.59,0.49){ $e_{20}$};
        \node at (1.865,1.35){ $e_{15}$};
        \node at (3.0,3.3){ $e_{8}$};
        \node at (1.792,2.5){ $e_{7}$};
        \node at (2.6,-3.55){ $e_{29}$};
        \node at (1.92,-2.75){ $e_{31}$};

        \node at (2.77,2.3){ $e_{14}$};
        \node at (4.0,2.4){ $e_{13}$};
        \node at (3.08,1.3){ $e_{16}$};
        \node at (4.0,1.6){ $e_{17}$};

        \node at (2.75,-1.65){ $e_{32}$};
        \node at (4.0,-2.4){ $e_{25}$};
        \node at (3.05,-2.8){ $e_{30}$};
        \node at (4.0,-1.6){ $e_{24}$};

        \node at (4.11,-0.53){ $e_{22}$};
        \node at (4.86,-1.51){ $e_{23}$};
        
        \node at (1.21,3){ $e_{6}$};
        \node at (-1.21,-3){ $e_{39}$};
        \node at (1.15,-1.05){ $e_{35}$};
        \node at (-1.21,1){ $e_{49}$};

      \node at (-1.2,3.3){\large $B_2$};
            \node at (-1.2,-0.9){\large $B_2^{-}$};
		\node at (1.1,0.75){\large $B_1^{-}$};
            \node at (1.1,-3.15){\large $B_1$};

		\end{tikzpicture}
	\end{center}
	\caption{Edges in the 3-ball $N$ for  $\Delta_{4,\infty,\infty;\infty}$. Since there are too many vertices and 2-cells, we only provide a few of their labels in this figure.
    Their full labels are the same as Figure \ref{figure:3ballfor4ppp}.}
	\label{figure:edgecircle4ppp}
\end{figure}

\section{A conjectural picture
\texorpdfstring{on $\Delta_{p,q,r;\infty}$ of type $A$}{} %
}
\label{sec:3mfdconj}

This section discusses related topics on 3-manifolds associated with complex hyperbolic triangle groups with accidental parabolic elements. 
This hopefully sketches a global conjectural picture of the 3-manifolds at infinity of the groups  $\Delta_{p,q,r;\infty}$ for $3 \leq p \leq 9$.

\subsection{3-manifold at infinity of
\texorpdfstring{$\Delta_{p,\infty,\infty;\infty}$ for $3 \leq p \leq 9$}{} %
}

In \cite{MaXie2021}, Ma-Xie proved the following
 	
\begin{thm} \label{thm:3pp}
	Let $\Gamma=\langle I_1, I_2, I_3 \rangle$ be the complex hyperbolic  triangle group $\Delta_{3,\infty,\infty;\infty}$. Then the 3-manifold at infinity of the even subgroup $\langle I_1I_2,I_2I_3\rangle$ of $\Gamma$ is the magic 3-manifold  in the Snappy  Census.	
\end{thm}

	The magic 3-manifold in the Snappy Census \cite{CullerDunfield:2014} is the complement of the chain link $6_1^3$ in the 3-sphere,  which is hyperbolic with three cusps and volume 5.3334895669 numerically. 
	
    Recall that we have proved that the 3-manifold at infinity of the even subgroup of $\Delta_{4,\infty,\infty;\infty}$ is the complement of the chain link $8^4_1$  in the 3-sphere. 
	
	Compared with the above results, we denote by $M_p$ the 3-manifold at infinity of the even subgroup of $\Delta_{p,\infty,\infty;\infty}$ for $3 \leq p \leq 9$. Recall that Schwartz's conjecture in \cite{schwartz-icm} bets this group is discrete, so $M_p$ is well-defined.
The complex hyperbolic group $\Delta_{p,\infty,\infty}$ is of type $A$ in the sense of R. Schwartz, therefore $p\le 9$. 
For more details, see \cite{schwartz-icm}.
We note that in the group $\Delta_{p,\infty,\infty;\infty}= \langle I_1,I_2,I_3 \rangle $ for $3 \leq p \leq 9$,  $I_1I_2$, $I_1I_3$ and $I_1I_3I_2I_3$ are parabolic by definition. But it can be checked directly $I_1(I_3I_2)^kI_3$ is also parabolic for any $2 \leq k \leq p-2$. Therefore, there are (at least) $p$ conjugate classes of parabolic elements in $\Delta_{p,\infty,\infty;\infty}$. In particular, there are  (at least) $p-3$ hidden conjugate classes of parabolic elements in  $\Delta_{p,\infty,\infty;\infty}$ which are unexpected. 
We propose the following conjecture.

\begin{conj}\label{conj:ninftyinftyinfty}	 For $3 \leq p \leq 9$, 
	the 3-manifold $M_p$ at infinity of the even subgroup of the complex hyperbolic triangle group $\Delta_{p, \infty,\infty;\infty}$ is the complement of the chain link $C_p$ in the 3-sphere. 
	\end{conj}

Here for each $p \geq 3$, the chain link $C_p$ is a hyperbolic link in the 3-sphere with $2p$ crossing number and $p$ components. 
See page 144 of \cite{Thurston} for the diagrams of chain links. We note that chain links $C_3$,  $C_4$ and $C_5$ are links  $6^3_1$, $8^4_1$ and $10^5_1$ respectively 	\cite{CullerDunfield:2014}. 
In Figure \ref{fig:841}, we show the diagrams of $C_4$ and $C_5$.
In other words, results in \cite{MaXie2021}  and this paper are special cases of Conjecture \ref{conj:ninftyinftyinfty}.
We also have very strong evidence that Conjecture \ref{conj:ninftyinftyinfty} is true for $n=5$, which will be proved in a subsequent paper.

\subsection{3-manifold at infinity of
\texorpdfstring{$\Delta_{p,q,r;\infty}$ for $3 \leq p \leq 9$}{} %
}

In \cite{MaXie2021}, Ma-Xie proposed
	\begin{conj}\label{conj:3mninfty}	
		The 3-manifold at infinity of the even subgroup of the complex hyperbolic  triangle group $\Delta_{3,m,\infty;\infty}$ is the hyperbolic 3-manifold obtained via the Dehn surgery of
		$6_1^3$ on the first  cusp with slope $m-2$. Moreover,
		the 3-manifold at infinity of the even subgroup of the complex triangle group $\Delta_{3,m,n;\infty}$ is the hyperbolic 3-manifold obtained by the Dehn fillings of $6_1^3$ on the first two cusps with slopes $m-2$ and $n-2$ respectively.
	\end{conj}

Note that in Conjecture \ref{conj:3mninfty} we use the meridian-longitude systems of the cusps of $S^3-6_1^3$ as in \cite{MartelliP:2006}, which is different from the meridian-longitude systems in Snappy.
	Results in \cite{MaXie2021, Acosta:2019, der-fal, MaXie2020, ParkerWill:2017} provided very strong evidence of Conjecture \ref{conj:3mninfty}.

In 	\cite{jwx}, Jiang-Xie-Wang proved
\begin{thm} \label{thm:44p}
	Let $\Gamma=\langle I_1, I_2, I_3 \rangle$ be the complex hyperbolic  triangle group $\Delta_{4,4,\infty;\infty}$. Then the 3-manifold at infinity of the even subgroup $\langle I_1I_2,I_2I_3\rangle$ of $\Gamma$ is the  3-manifold $s782$ in the Snappy  Census.	
\end{thm}

It is not difficult to show that $s782$  can be obtained by Dehn fillings  on two  cusps of  $S^3-8^4_1$.  So, the 3-manifold at infinity of $\Delta_{4,4,\infty;\infty}$ can be obtained from the 3-manifold at infinity of $\Delta_{4,\infty,\infty;\infty}$  by Dehn fillings on two cusps.

Motivated by all above, we also propose 
	\begin{conj}\label{conj:pqrinfty}	For each $3 \leq p \leq 9$, there is a preferred  thrice-cusped hyperbolic 3-manifold $N_p$, which is obtained by Dehn filling the $p-3$ cusps of $M_p$. Moreover,  the 3-manifold at infinity of the even subgroup of the complex hyperbolic triangle group $\Delta_{p,q,r;\infty}$ is obtained by Dehn filling two cusps of  $N_p$ when  $p <q< r$.
\end{conj}

Here, we assume $p <q< r $, so there is no obvious symmetry of the group $\Delta_{p,q,r;\infty}$. 
From the group  $\Delta_{p,\infty,\infty;\infty}$ to the group   $\Delta_{p,q,r;\infty}$, we first kill  $p-3$ hidden conjugate classes of parabolic elements. 
This process corresponds to  topological Dehn fillings from $M_p$ to $N_p$. Then two more Dehn fillings on $N_p$ should result in the 3-manifold  at infinity of the even subgroup of $\Delta_{p,q,r;\infty}$. 
In particular, the manifolds $M_3$ and $N_3$ are the same.
So in other words, for $3 \leq p \leq 9$, the guessed 3-manifold $M_{p}=S^3-C_p$ should be a seed which will generate many explicit 3-manifolds with spherical CR uniformizations. 
The authors have a weak guess of what the 3-manifold $N_p$ should be, but massive computations are missing.

\end{document}